%% file: main-reviewSICON.tex
\documentclass[12pt]{article}
\usepackage{color}
\usepackage{latexsym,dsfont}
\usepackage{amssymb, stmaryrd}
\usepackage{amsmath}
\usepackage{amsthm}
\usepackage{enumerate}
\usepackage{hyperref}
\usepackage{comment}
\usepackage{graphicx,float}
\usepackage{caption}
\usepackage{comment}
\usepackage{subcaption}
\usepackage{setspace}
\usepackage{xcolor}
\usepackage{diagbox}
\usepackage{amsfonts, amscd}
\usepackage{mathrsfs}
\usepackage{mathtools} 

\usepackage[round]{natbib}
\setcitestyle{authoryear}

\input{notation}

\title{Optimal Stopping of Branching Diffusion Processes}
\author{ Idris Kharroubi\footnote{Research of the author partially supported by ANR grant RELISCOP.}\\ LPSM, UMR CNRS 8001,\\ Sorbonne Universit\'e
  and Universit\'e Paris Cit\'e,
  \\ idris.kharroubi @ sorbonne-universite.fr
  \and Antonio Ocello
  \footnote{The work of this author was funded by the European Union (ERC-2022-SYG-OCEAN-101071601). Views and opinions expressed are however those of the author only and do not necessarily reflect those of the European Union or the European Research Council Executive Agency. Neither the European Union nor the granting authority can be held responsible for them.}
  \\ 
  CMAP, UMR CNRS 7641,\\
  Ecole Polytechnique,
  \\ antonio.ocello @ polytechnique.edu
  }
\date{\today}

\begin{document}

\maketitle

\begin{abstract}
  This article explores an optimal stopping problem for branching diffusion processes.  It consists in looking for optimal stopping lines, a type of stopping time that maintains the branching structure of the processes under analysis. By using a dynamic programming approach, we characterize the value function for a multiplicative cost, which may depend on the particle's label. We reduce the problem's dimensionality by setting a branching property and defining the problem in a finite-dimensional context. Within this framework, we focus on the value function, establishing uniform continuity and boundedness properties, together with an innovative dynamic programming principle. This outcome leads to an analytical characterization with the help of a nonlinear elliptic PDE. We conclude by showing that the value function serves as the unique viscosity solution for this PDE, generalizing the comparison principle to this setting.
\end{abstract}

\noindent \textbf{MSC Classification- }60G40, 60J80, 35J60, 49L20, 49L25

\noindent \keywords{
  Optimal stopping, branching diffusion process, dynamic programming principle, Hamilton-Jacobi-Bellman equation, viscosity solution.
}


\maketitle

\section{Introduction}

Since its introduction in the late sixties in \citet{INW69, INW681, INW682, Sk64}, the class of branching diffusion processes received a great deal of interest. This object is used to describe the evolution of a population where we are interested in a special feature, $e.g.$, the spatial motion, of identical particles that reproduce at random times.

These processes are well-suited  in capturing a dual level of interaction. A macroscopic dynamics, marked by the branching aspect, is connected to a microscopic one, characterized by a stochastic differential equation. By establishing a link between macroscopic and microscopic facets, these dynamics prove applicable in a wide array of domains, from biology to finance. In the realm of biology, they prove invaluable for modeling phenomena such as parasite infection within cell populations  \citep[see, $e.g.$,][]{bansaye2010branching, marguet2020parasite, marguet2023parasite}. Conversely, in the financial domain, these processes are used to characterize options related to cryptocurrencies \citep[see, $e.g.$,][]{kharroubi2024stochastic}.

In the study of branching diffusion processes, a fundamental question emerges: at what juncture does it become optimal to halt such a process? This question delves into the determination of an opportune point in time to stop the evolution of a branching diffusion. This research line echoes the optimization of a given functional to trade-off between the diffusion and reproduction of these processes and a possible degradation of the reward. By investigating the optimal stopping time for branching diffusion processes, we aim to shed light on the decision-making process involved in terminating these dynamical systems, thereby enhancing our understanding of their behaviour and enabling more effective applications in various fields of study.

One possible approach to consider is looking at the entire branching diffusion process as a whole, as done in \citet{Ocello:rel_form_branching}, and finding a universal stopping time that applies to all active branches simultaneously. This global stopping time serves as a comprehensive decision rule, enabling a synchronized halt to the progression of each branch in the system, regardless of their characteristics or temporal disparities.

Although the aforementioned approach has its appeal, it may not fully align with the intrinsic structure of such processes. Indeed, the fundamental nature of a branching process, even when studied as a collective entity, is fundamentally rooted in its ability to portray the trajectory and dynamics of a singular individual. Therefore, while a global perspective may offer valuable insights and provide a comprehensive overview of the system, it may inadvertently disregard the inherent individuality of the branches.

This dual mode between the individuality of the single component as opposed to the wholeness of the population is a key concept in cooperative game theory. For example, mean-field control literature \citep[see, $e.g.$,][]{book:Carmona-Delarue_1, book:Carmona-Delarue_2} deals with the control of large-scale systems involving a multitude of interacting agents, assumed to be rational decision-makers who aim to optimize their objective functions. The goal is to find control strategies that maximize a specific objective at the population level, which aligns with the optimal behaviour of each agent, influenced by the collective behaviours of the entire population.
An additional example illustrating the transformation of global behaviour into individual optimization can be observed in \citet{claisse18, kharroubi2024stochastic, Ocello:Ctrl-superprocesses, Ocello:rel_form_branching}. These studies prove how control strategies are contingent upon the decisions made by each participant. Moreover, the concept of the branching property emerges as a means to reduce the complexity of the problem, consequently shifting the focus toward analyzing the dynamics of the individual agents.

To capture the decision-making process of individuals within a collective framework, we adopt the concept of stopping lines. This mathematical object, introduced in \citet{Chauvin-86-1,Chauvin-86-2}, serves as the counterpart to stopping times in branching dynamics.
Stopping lines are characterized by a subset of the process's genealogy, where no member can be traced back to another member, and we can see their use in applications such as \citet{Lambert:Covid:efficiency_quarantining}.

Although stopping lines have been used in previous studies, the exploration of optimal stopping lines based on specific criteria remains, to the best of our knowledge, an open problem. This article aims to address this research gap by directing our attention to this exact issue.

An application of this optimal stopping problem is possible in the field of finance, specifically in the valuation of American options tied to cryptocurrencies \citep[see, $e.g.$,][]{option-cruptocurrencies}. This modeling is the one discussed in \citet{kharroubi2024stochastic}, in the case of super-replication.
In the realm of biology, another pertinent utilization arises in the optimization of halting infections caused by parasites, a model explored by \citet{bansaye2010branching, marguet2020parasite, marguet2023parasite}. 
This application gains significance as the initial stage of a mutant invasion closely aligns with the characteristics of a branching process, as expounded in works like \citet{barbour2013approximating, barbour2015escape, bansaye2022sharp}. In this phase, critical stopping criteria become imperative, especially in the context of identifying the emergence and detection of cancer.

Within a branching diffusion process framework, we look for the characterization of the value function linked to an infinite horizon optimal stopping problem. Optimization is done over the set of stopping lines,
where each branch becomes eligible for halting only if no preceding ancestor has been stopped before.
We narrow our investigation to multiplicative rewards, similar to the approach taken in \citet{claisse18, Nisio}. Drawing inspiration from \citet{Chauvin91}, we prove a fundamental branching property. This enables us to analyze systems originating from a single particle and to provide a differential characterization of the corresponding problem.

Moreover, this form of the reward function has been utilized in the financial literature to leverage branching processes for the development of numerical methods aimed at solving semilinear PDEs \citep[see, e.g.,][]{henry2012counterparty, henry2019branching, henry2012cutting, agarwal2020branching, henry2021branching}. By applying our framework to the specific case where all the reward functions are equal, we align with this established body of work, particularly in the context of semilinear obstacle problems. In Remark \ref{rmk:g_i_equal-1} and Remark \ref{rmk:g_i_equal}, we explicitly demonstrate how our approach encompasses this setting, providing a clear pathway for connecting our results to this literature. This connection underscores the versatility of our approach, demonstrating its relevance not only as a general theoretical contribution but also as a practical tool within the numerical analysis literature for such PDEs.

Establishing an original Dynamic Programming Principle (DPP), we extend the framework of the classical optimal stopping problem to our branching context.
This outcome paves the way for an analytical characterization of the value function.

The corresponding PDE takes the form of an obstacle problem with a semilinear term, which involves a polynomial series associated with the branching mechanism and value functions related to offspring labels. Under the assumption that this series has an infinite radius of convergence, we show that the value function is a solution in the sense of viscosity to this PDE. 

To conclude the PDE characterization, we present a comparison theorem. The presence of the semilinear term, tied to the value functions associated with offspring labels, introduces a non-classical aspect to this PDE. 
We explore a multiplicative penalization, making the viscosity solutions go towards zero in the spatial variable as a result of the previously demonstrated polynomial growth. Then, using the assumption of vanishing rewards as the label goes to infinity, we establish the comparison principle for value functions related to sufficiently large starting labels. We finally extend this analysis to cover the remaining functions through a backward induction on the size of the label.

The remainder of the paper is structured as follows. Section \ref{Section:branching-diffusions} presents a detailed description of the model under examination, focusing on the characteristics of branching diffusion processes. In Section \ref{Section:optimal-stopping-pb}, we introduce the notion of stopping lines and describe the optimal stopping problem. Section \ref{Section:DPP} is dedicated to the regularity properties of the value function and to the dynamic programming principle. In Section \ref{Section:HJB}, we characterize the value function as the unique viscosity solution to an obstacle problem. Finally, Sections \ref{sec:thm:DPP} and \ref{sec:theorem:result_PDE} are devoted to the proofs of the main results, the dynamic programming principle and the viscosity solution properties respectively.

\section{Branching diffusion processes formulation}\label{Section:branching-diffusions}

\paragraph{Notations.}
We adopt the Ulam--Harris notation. Let $\Ic$ be the set of labels
\begin{align*}
    \Ic := & \{\varnothing\}\cup\bigcup_{n=1}^{+\infty}\N^n\;,
\end{align*}
where the label $\varnothing$ corresponds to the mother particle. For $n\geq 1$, write $i=i_1\ldots i_n$ for the multi-integer $i=(i_1,\ldots,i_n)\in\N^n$ and, for $n,m\geq 1$, $i=i_1\ldots i_n\in \N^n$, and $j=j_1\ldots j_m\in \N^m$, define their concatenation $i j\in\N^{n+m}$ as
\begin{align*}
    i j := &  i_1\ldots i_n j_1\ldots j_m\;,
\end{align*}
with $\varnothing i = i \varnothing=i$, for all $i\in\Ic$. Consider the partial ordering $\preceq$ (resp. $\prec$) by
\begin{align*}
    i\preceq j ~ \Leftrightarrow ~ \exists \ell\in\Ic~:~i=j\ell
    \qquad
    \left(\textrm{resp.}~i\prec j ~ \Leftrightarrow ~\exists \ell\in\Ic\setminus \{\varnothing\}~:~i=j\ell\right)\;,
\end{align*}
for $i,j\in\Ic$.
We endow $\Ic$ with the discrete topology, generated by the distance
\begin{align*}
    d^\Ic(i,j) := \sum_{\ell = p+1}^n (i_\ell +1) + \sum_{\ell' = p+1}^m (j_{\ell'} +1)\;,
    \qquad\text{ for }
    i=i_1\cdots i_{n}\in\N^n, \;j=j_1\cdots j_{m} \in \N^m\;,
\end{align*}
with $p = \max\{\ell\geq1:i_\ell=j_\ell\}$ the generation of the greatest common ancestor, and write $|i|:= d^\Ic(i,\varnothing)$ for $i\in\Ic$.
For $i\in\Ic\setminus\{\varnothing\}$, $i=i_1\cdots i_n\in\N$, denote by $i-$ its parent, defined by $i-=i_1\cdots i_{n-1}$.

Let $\D\l(\R_+;\Sc\r)$ be the space of càdlàg (right-continuous and left-limited) functions that take values in $\Sc$ and $\Mc(\Sc)$ the set of finite measure on $\Sc$ endowed with the weak topology, for $\Sc$ a Polish space. 
Let $E$ be the subset of $\Mc\l(\Ic\times\R^d\r)$ defined as
\begin{align*}
    E:=\l\{
      \sum_{i\in V}\delta_{(i,x_i)}\;:\; V\subseteq \Ic \text{ finite},\; x_i\in\R^d,\; i\nprec  j,\text{ for }i,j\in\Vc
    \r\}\;.
\end{align*}
As it is a closed set in $\Mc\l(\Ic\times\R^d\r)$ \citep[see, $e.g.$, Proposition A.7,][]{kharroubi2024stochastic}, we have that $E$ is Polish \citep[see, $e.g.$, Section 3.1.1,][]{Dawson}.
Denote $\langle \mu,f\rangle := \sum_{i\in V} f_i\l(x_i\r)$ for $\mu:=\sum_{i\in V} \delta_{(i,x_i)}$ and $f=\l(f_i\r)_{i\in\Ic}$ with $f_i:\R^d\to\R$. Take $b:~ \R^d \rightarrow \R^d$ and $\sigma:~ \R^d \rightarrow \R^{d \times m}$ the drift term and diffusion coefficient of the spatial motion of the particles of the population, and consider spatially dependent death intensity $\alpha:~ \R^d \rightarrow \R_+$ and reproduction probabilities $p_k: \R^d \rightarrow [0,1]$, for $k \in \N$, with $\sum_{k\in\N}p_k(x) = 1$, for $x\in \R^d$.
Denote $\Lc$  the infinitesimal generator of the spatial motion
\begin{align*}
    \Lc f(x):=\frac{1}{2}{\rm Tr}\l(\sigma\sigma^\top(x)D^2_x f(x)\r) + b(x)^\top D_x f(x)\;,\qquad \text{ for }f \in C^2(\R^d)\;,
\end{align*}
where $D^2_x f$ and $D_x f$ denote respectively the Hessian matrix and the gradient of $f$.

\paragraph{Canonical space.} We focus on constructing the branching diffusion process to be studied within its canonical space as done in \citet{claisse18-v1}.

Fix a real number $\bar\alpha>0$.
Consider the space $C^d:= C\l(\R_+,\R^d\r)$ of continuous paths (resp. $M$ of integer-valued Borel measures on $\R_+\times \l[0,\bar\alpha\r]$ that are
locally finite), equipped with the locally uniform convergence metric (resp. equipped with the vague topology). Note that $M$ with the vague topology is Polish \citep[see, $e.g.$, Theorem 4.2,][]{book:KALLENBERG-RM}. Take on $C^d$ (resp. $M$) its canonical filtration $\l(\Cc^d_s\r)_{s\geq0}$ (resp. $\l(\Mc_s\r)_{s\geq0}$) and its Borel $\sigma$-algebra $\Cc^d$ (resp. $\Mc$), which coincides with $\sigma\l(\bigvee_{s\geq0}\Cc^d_s\r)$ \citep[see, $e.g.$, Section 1.3,][]{SV97} (resp. with $\sigma\l(\bigvee_{s\geq0}\Mc_s\r)$). Note that $\Mc_s$ is the smallest $\sigma$–algebra that makes the map $\mu\in M\mapsto \mu(A)$ measurable for any Borel subset $A$ of $[0,s]\times [0,\bar \alpha]$.

Define the space $H$ (resp. $H^i$ for $i\in\Ic$), its Borel $\sigma$-algebra $\Hc$ (resp. $\Hc^i$) and its filtration $\l(\Hc\r)_{s\geq0}$ (resp. $\l(\Hc^i\r)_{s\geq0}$) as
\begin{align*}
    H &:= (C^d\times M)^{\Ic}\;\;,
    \qquad\qquad~~~~~
    \Hc := (\Cc^d\times \Mc)^\Ic\;,
    \\
    \text{(resp. }
    H^i&:=(C^d\times M)^{\{j\in\Ic~:~j\preceq i\}}\;,
    \qquad
    \Hc^i:=(\Cc^d\times \Mc)^{\otimes\{j\in\Ic~:~j\preceq i\}}
    \text{ )}
    \;,
\end{align*}
and
\begin{align*}
    \Hc_s := \l(\Cc_s\times \Mc_s\r)^{\Ic}
    \qquad
    \text{(resp. }
    \Hc^i_s:=(\Cc^d_s\times \Mc_s)^{\otimes\{j\in\Ic~:~j\preceq i\}}
    \text{ )}
    \qquad\text{ for }s\geq 0\eqsp.
\end{align*}

Define, now, the canonical filtered probability space $\big(\Omega,\Fc,\F=\l(\Fc_s\r)_{s\in\R_+},\P\big)$ as
\begin{align*}
    &\Omega := H\;,
    \qquad
    \Fc := \Hc^\P\;,
    \qquad
    \Fc_s := \Hc^\P_s\;,
    \qquad
    \P := (\P^\circ\otimes\P^{\i})^{\otimes\Ic}
    \;,
\end{align*}
where $\Hc^\P$ (resp. $\l(\Hc^\P_s\r)_{s\geq0}$) is the usual $\P$-completion of $\Hc$ (resp. $\l(\Hc_s\r)_{s\geq0}$) and $\P^\circ$ (resp. $\P^{\i}$) is the Wiener measure on $C^d$ (resp. the Poisson measure on $\R_+\times[0,\bar\alpha]$ with Lebesgue intensity). In this space, define the canonical processes $\l(B^i,Q^i\r)_{i\in\Ic}$ as
\begin{align*}
    B^i\l(\l(w_j,\mu_j\r)_{j\in\Ic}\r)=w^i_s\;,\qquad
    Q^i\l(\l(w_j,\mu_j\r)_{j\in\Ic},\;A\r)=\mu^i(A)\;,
\end{align*}
for $\l(w_j,\mu_j\r)_{j\in\Ic}\in \Omega$, $s\geq0$, and $A\in\Bc\l(\R_+\times \l[0,\bar\alpha\r]\r)$. From their definitions, under $\P$, these maps represent a family of independent random variables such that $B^i$ is a $m$-dimensional Brownian motion and $Q^i\l(\rmd t,\rmd z\r)$ is a Poisson random measure on $\R_+\times \R_+$ with Lebesgue intensity measure. Note that
\begin{align*}
    \Hc^i_s & =  \sigma\Big(B^j_s, Q^j\l(A\r)~:~  j\preceq i,\; s\in[0,t],\; A\in \Bc\l([0,t]\times \R_+\r)\Big)
\end{align*}
and
\begin{align*}
    \Hc^i & = \sigma\l(\bigvee_{s\geq0}\Hc^i_s\r)\eqsp\eqsp,\qquad\text{ for }i\in\Ic\eqsp.
\end{align*}

For a random variable $Y$, define the shifted random variable $Y^{t,\bar \omega}$ as
\begin{align*}
    Y^{t,\bar \omega}(\omega) = Y(\bar \omega \oplus_t \omega)\;,
\end{align*}
with $\bar \omega \oplus_t \omega$ defined by
\begin{align*}
    B^i\l(\bar \omega \oplus_t \omega\r)_s ~ = ~ \bar w^i_{s\wedge t}+(w^i_{s\vee t}-w^i_{t})
    \quad
    \text{ and }
    \quad
    Q^i\l(\bar \omega \oplus_t \omega\r)~=~\restr{\bar \mu^i}{[0,t]}+\restr{\mu^i}{(t,+\infty]}
    \eqsp,
\end{align*}
for $t\geq 0$ and  $\bar \omega=(\bar w^i,\bar \mu^i)_{i\in\Ic}, \omega=(w^i,\mu^i)_{i\in\Ic}\in\Omega$

Working in the canonical space simplifies the general presentation of the results. An alternative strong construction of the process $Z^\mu$ that satisfies \eqref{eq:semimartingale-decomposition} can be linked to this construction in the canonical space without loss of generality, with the use of \citet[Proposition 3.5 and Corollary 3.6,][]{claisse18-v1}.
 
\paragraph{Existence and path-wise uniqueness.}

  As in \citet{claisse18}, let $Z^\mu$ be the following measure valued process
  \begin{align*}
    Z^\mu_t := \sum_{i\in\Vc^\mu_t}\delta_{\l(i,X^{\mu,i}_t\r)}\;,
  \end{align*}
  starting at $t=0$ from the initial configuration $\mu\in E$, with $\Vc^\mu_t$ the set of labels of the alive particles and $X^{\mu,i}_t$ the position of the particle $i$, for $t\geq0$. The dynamics of a process $Z^{\mu}$ is characterized by the following family of semi-martingale decomposition
\begin{align}
\label{eq:semimartingale-decomposition}
    \begin{split}
        \l\langle Z^\mu_{t},f\l(t,\cdot\r)\r\rangle
        =&
        \l\langle \mu,f\l(0,\cdot\r)\r\rangle
        +
        \int_0^{t}\sum_{i\in\Vc^\mu_{s}} D_x f_i\l(s,X^{\mu,i}_s\r)\sigma\l(X^{\mu,i}_s\r)dB^i_s
        \\
        &+\int_0^{t}\sum_{i\in\Vc^\mu_{s}} \l(
        \partial_s f_i\l(s,X^{\mu,i}_s\r)+\Lc f_i\l(s,X^{\mu,i}_s\r)\rmd s
        \r)
        \\
        &+\int_{(0,t]}\sum_{i\in\Vc^\mu_{s-}} 
        \sum_{k\geq0}
        \l(
        \sum_{\ell=0}^{k-1} f_{i\ell}\l(s,X^{\mu,i}_s\r)- f_i\l(s,X^{\mu,i}_s\r)
        \r)\1_{I_k\l(X^{\mu,i}_s\r)}(z)Q^i(\rmd s\rmd z),
        \\
        &\qquad\qquad\qquad\qquad\qquad\qquad\qquad\qquad\qquad\qquad
        \text{ for }t\geq0,\P(d\omega)-a.s.,
    \end{split}
\end{align}
for $f=\l(f_i\r)_{i\in\Ic}\in C^{1,2}\l(\R_+\times \R^d\r)^\Ic$, where,
\begin{align*}
    \begin{split}
      \Vc^\mu_s:=&\l\{i\in\Ic~:~Z_s\l(\{i\}\times \R^d\r)\neq 0\r\}\;,
      \\
      I_k\l(x\r) :=& \l[
        \alpha(x)\sum_{\ell=0}^{k-1}p_\ell(x),\;
        \alpha(x)\sum_{\ell=0}^{k}p_\ell(x)
      \r)\;,
    \end{split}\qquad\text{ for }(s,x)\in\R_+\times\R^d\;.
\end{align*}

In \eqref{eq:semimartingale-decomposition}, one can see that the first two integrals represent the spatial diffusive behaviour of the particles, while the integral with respect to the Poisson random measures captures the jumps in the measure-valued process $Z$ caused by particle death and reproduction. Additionally, this term demonstrates that reproduction in this process is \textit{local}, meaning the offspring are born at the exact location where their mother dies.

More general interactions in the model parameters can be considered, allowing for dependence on the coefficients \citep[see, $e.g.$,][]{Ocello:rel_form_branching}.

Consider the following assumptions, which will be used throughout the paper.
\begin{Assumption}
\label{Assumption_H_0}
    \textit{(i)}
    The functions  $b$ and $\sigma$ are Lipschitz continuous, $i.e.$, there exists a constant $L>0$ such that
    \begin{align}
    \label{eq:Lipschitz_assumption:b_sigma}
        |b(x)-b(x')|+|\sigma(x)-\sigma(x')|  \leq & \;L\;|x-x'|\;,
        \qquad
        \text{ for } x,x'\in\R^d\;.
    \end{align}

    \textit{(ii)}
    The functions $(p_k)_{k\geq 0}$ measurable and satisfy
    \begin{align*}
      M:=\sup_{x\in\R^d}\sum_{k\geq0}k \;p_k(x)\;<\;\infty\;.
    \end{align*}

    \textit{(iii)}
    The function $\alpha$ is measurable and bounded by $\bar\alpha>0$.
\end{Assumption}

These assumptions ensure existence and path-wise uniqueness of the process $Z$, corresponding to the uncontrolled version of \citet[Proposition 2.1,][]{claisse18}, as stated below. Denote $N^\mu$ the process corresponding to the number of alive particles, $i.e.$,
\begin{align*}
    N^\mu_t & =  \l|\Vc^\mu_t\r|\;,\quad t\geq 0\;,
\end{align*}
where $|V|$ stands for the cardinality of a set $V\subset \Ic$.

\begin{Proposition}
    \label{Prop:existence-uniqueness}
    Fix $\mu=\sum_{i\in V}\delta_{(i,x_i)}\in E$. Suppose that Assumption \ref{Assumption_H_0} holds. Then, there exists a unique (up to indistinguishability) càdlàg and adapted process $\l(Z^{\mu}_{t}\r)_{t\geq0}$ valued in $E$ satisfying the semi-martingale decomposition \eqref{eq:semimartingale-decomposition}. Moreover, we have
    \begin{align*}
        \E\l[
        \sup_{0\leq s\leq t}\;
        N_s^{\mu}
        \r] &\leq
        \l|V\r| {\rm e}^{\bar\alpha M t}\;,
        \qquad
        \text{ for }t\geq0\;.
    \end{align*}
\end{Proposition}

\paragraph{Exponential moments of the number of particles.}
In the sequel, we shall require the process $ N^\mu $ to have finite exponential moments, a condition necessary to ensure that the problem under consideration is well-defined.
   
\begin{Assumption}
\label{Assumption_H_0BIS}
    The functions $(p_k)_{k\geq 0}$ admits a finite moment of any order, $i.e.$,
    \begin{align*}
        M_\ell & := \sup_{x\in\R^d}\sum_{k\geq0}k^\ell \;p_k(x)~<~\infty\;,\qquad\text{ for }\ell\geq 1\eqsp,
    \end{align*}
    and 
    \begin{align*}
        \bar M:=\sup_{\ell\geq0}2^\ell M_\ell  <  +\infty\eqsp.
    \end{align*}
\end{Assumption}
  
This condition ensures that the process $N^\mu$ possesses exponential moments, which are fundamental for the development of the problem. Moreover, a typical example of a family $(p_k)_{k\geq0}$ satisfying Assumption \ref{Assumption_H_0BIS} is given by a Poisson distribution of the form
\begin{align*}
    p_k(x) &= \frac{\lambda(x)^k}{k!}{\rm e}^{-\lambda(x)}\;,\quad x\in\R^d\;,
\end{align*}
with $\lambda:~\R^d\rightarrow\R_+$ a Borel function such that
\begin{align*}
  \lambda(x) & \leq \frac{1}{2}\;,\qquad\text{ for }x\in\R^d\eqsp. 
\end{align*}
For a given initial condition $\mu\in E$, let $\bar N^\mu$  be the process corresponding to the number of all particles that have been alive up to the present time, $i.e.$,
\begin{align}
\label{defProcN}
    \bar  N^\mu_t &= \l|\bigcup_{s\in[0,t]}\Vc^\mu_s\r|\;,\qquad\text{ for }t\geq 0\;.
\end{align}
Under the aforementioned assumption, we can establish exponential moment estimates for this process as follows.

\begin{Proposition}
\label{propoestimN}
    Fix $\mu=\sum_{i\in V}\delta_{(i,x_i)}\in E$ and suppose that Assumptions \ref{Assumption_H_0} and \ref{Assumption_H_0BIS} hold. Then, we have
    \begin{align*}
      \E\l[
        K^{\bar N_s^{\mu}}
      \r] &\leq
      (K\vee 1)^{|V|\rme^{\bar \alpha \bar M t}}\;,\qquad\text{ for }t\geq 0,K>0\;.
    \end{align*}
\end{Proposition}
\begin{proof}
    We first suppose that $K>1$ as the result is obvious for $K\leq1$. Let $(\theta_n)_{n\geq 1}$ be the sequence of stopping times defined by
    \begin{align*}
    \theta_n &= \inf\big\{s\geq 0~:~N^\mu_s\geq n\big\}\;,\qquad \text{ for }n\geq 1\;.
    \end{align*}
    From the dynamics \eqref{eq:semimartingale-decomposition} of the process $Z^\mu$, we have 
    \begin{align*}
        N^\mu_{\tau_n\wedge t} &= |V|+\int_{(0,\tau_n\wedge t]\times[0,\bar\alpha]}\sum_{i\in \Vc_{s-}^\mu} \sum_{k\geq0}
        (k-1)\1_{I_k\l(X^{\mu,i}_s\r)}(z)Q^i(\rmd s\rmd z)\;,\qquad \text{ for }t\geq0\;.
    \end{align*}
    This equation yields that when we look at $\bar N^\mu$, we have
    \begin{align*}
        \bar N^{\mu}_{\tau_n\wedge t} &= |V|+\int_{(0,\tau_n\wedge t]\times[0,\bar\alpha]}\sum_{i\in \Vc_{s-}^\mu} \sum_{k\geq0}
        k\1_{I_k\l(X^{\mu,i}_s\r)}(z)Q^i(\rmd s\rmd z)\;,\qquad \text{ for }t\geq0\;.
    \end{align*}
    Using Itô's formula for jump processes, we get
    \begin{align*}
        \l|\bar N^\mu_{\tau_n\wedge t}\r|^\ell &= |V|^\ell+\int_{(0,\tau_n\wedge t]\times[0,\bar\alpha]}\sum_{i\in \Vc_{s-}^\mu} \sum_{k\geq0}
         \l(|\bar N^\mu_{s-}+k|^\ell-|\bar N^\mu_{s-}|^\ell\r)\1_{I_k\l(X^{\mu,i}_s\r)}(z)Q^i(\rmd s\rmd z),
    \end{align*}
    for $\ell\geq1$ and $t\geq0$. Therefore, taking the expectation and using the definition of the intervals $I_k$, we get
    \begin{align*}
        \E\l[ 
        \l|\bar N^\mu_{\tau_n\wedge t}\r|^\ell\r]
        &= 
        |V|^\ell+
        \E\l[
            \int_{0}^{\tau_n\wedge t}\sum_{i\in \Vc_{s}^\mu} \alpha(X^{\mu,i}_s)\sum_{k\geq0}
            \l(\l|\bar N^\mu_{s}+k\r|^\ell-|\bar N^\mu_{s}|^\ell\r)p_k\l(X^{\mu,i}_s\r)\rmd s
        \r]
        \eqsp.
    \end{align*}
    Using the inequality
    \begin{align*}
    a^\ell-b^\ell = & (a-b)\sum_{p=0}^{\ell-1} a^p b^{\ell-1-p}\leq (a-b)\ell a^{\ell-1}\;,\qquad \text{ for }a\geq b\geq 1\;,
    \end{align*}
    we get
    \begin{align*}
        \E\l[ 
            \l|\bar N^\mu_{\tau_n\wedge t}\r|^\ell
        \r]
        &\leq
        |V|^\ell+
        \E\l[
            \int_{0}^{\tau_n\wedge t}
            \ell
            \sum_{i\in \Vc_{s}^\mu} \alpha(X^{\mu,i}_s)
            \sum_{k\geq0}
            k\eqsp \l|\bar N^\mu_{s}+k\r|^{\ell-1}\eqsp p_k\l(X^{\mu,i}_s\r)\rmd s
        \r]
        \eqsp.
    \end{align*}
    Then using the inequality 
    \begin{align*}
    (a+b)^{n} & \leq 2^{n}(a^n+b^n) \;,\qquad \text{ for }a,b\in\R_+,~n\geq 1\;,
    \end{align*}
    and Assumption \ref{Assumption_H_0BIS}, we get
    \begin{align*}
        \E\l[
            \l|\bar N^\mu_{\tau_n\wedge t}\r|^\ell
        \r]
        &\leq
        |V|^\ell+
        \E\l[
            \int_{0}^{\tau_n\wedge t}
            \ell 2^{\ell-1}
            \l|\bar N^\mu_{s}\r|^{\ell-1}
            \sum_{i\in \Vc_{s}^\mu} \alpha(X^{\mu,i}_s)
            \sum_{k\geq0}
            k^{\ell}\eqsp p_k\l(X^{\mu,i}_s\r)\rmd s
        \r]
        \\
        &\leq
        |V|^\ell+
        \E\l[
            \int_{0}^{\tau_n\wedge t}
            \ell 
            \l|\bar N^\mu_{s}\r|^{\ell}
            \bar \alpha
            2^{\ell-1} M_\ell \rmd s
        \r]
        \\
        & \leq
        |V|^\ell+
        \E\l[
            \int_{0}^{t}\ell \bar \alpha \bar M \l|\bar N^\mu_{ \tau_n\wedge s}\r|^\ell \rmd s
        \r]
        \eqsp,
    \end{align*}
    for $t\geq0$.
    Using Grönwall's lemma and sending $n$ to $\infty$, we get from the monotone convergence theorem
    \begin{align*}
        \E\l[  \l|\bar N^\mu_{ t}\r|^\ell\r] & \leq |V|^\ell \rme^{\ell \bar \alpha \bar M t}
        \;,\qquad \text{ for }t\geq0,~\ell\geq1\;,
    \end{align*}
    We now have
    \begin{align*}
        \E\l[ K^{\bar N_t^{\mu}}
        \r] &=
        \E\l[ \rme^{\l(\bar N_t^{\mu}\r)\log(K)}
        \r]
        =
        \sum_{\ell\geq 0}
        \frac{1}{\ell!}{\log(K)^\ell}\E\l[  \l|\bar N^\mu_{ t}\r|^\ell\r]\\
        & \leq \sum_{\ell\geq 0}  \frac{1}{\ell!}{\log(K)^\ell}     |V|^\ell \rme^{\ell \bar \alpha \bar M t}\\
        & \leq K^{|V|\rme^{\bar \alpha \bar M t}}
        \eqsp,
    \end{align*}
    for $t\geq0$.
\end{proof}

\paragraph{Time of deaths.}
  Fix an initial configuration $\mu=\sum_{i\in V}\delta_{(i,x_i)}\in E$. Define inductively, from parent to child, the death-times $\l(S^\mu_i\r)_{i\in\Ic}$ as 
\begin{align*}
    S^\mu_i& := \inf\Big\{
      s>0~:~
      i\,\notin\Vc^\mu_s
    \Big\}\;,\qquad\text{ for }i \text{ such that }  i\preceq j\text{ for some } j\in V  \;,
\end{align*}
and
\begin{align*}
    S^\mu_i:=&\inf\Big\{
      s>S^\mu_{i-}~:~
      i\,\notin\Vc^\mu_s
    \Big\}\;,\qquad\text{ for others }i\in \Ic\;.
\end{align*}

From this definition, a particle $i$ that never borns, satisfies $S^\mu_i=S^\mu_{i-}$. In particular, we have two kinds of such particles : either the particle have a last ancestor $k$ that have been alive and we have by a backward induction $S^\mu_i=S^\mu_k$ or all the ancestors have never been alive and we have $S^\mu_i=0$.
Using the convention $\inf\emptyset=+\infty$, we see that the particles $i$ that are never died have automatically death-time equal to $S^\mu_i:=+\infty$ $\P$--a.s.

\section{The optimal stopping problem}\label{Section:optimal-stopping-pb}

\subsection{Stopping lines}

We now consider a notion of stopping times tailored to the branching nature of the process $Z$. This concept is known as \textit{stopping lines} and has been introduced in \citet{Chauvin-86-1,Chauvin-86-2}. As the branching structure depends on the process itself in our framework, we need to adapt the definition.
\begin{Definition}
    Fix $\mu\in E$. We say that a collection of maps  $\l(\tau_i\r)_{i\in\Ic}$ is a \emph{stopping line for initial condition $\mu$} if
    \begin{enumerate}[(i)]
        \item $\tau_i$ is an $(\Fc^i_t)_{t\geq0}$-stopping time for $i\in\Ic$ with $\tau_i\geq S^\mu_{i-}$,
        
        \item the random set $L^\mu_\tau$, defined by
        \begin{align*}
            L^\mu_\tau := &
            \Big\{ 
              i\in\Ic~:~
              S^\mu_{i-}\leq \tau_i < S^\mu_{i}
            \Big\}=\Big\{i\in\Ic~:~i\in\Vc^\mu_{\tau_i}\Big\}\;,
        \end{align*}
        satisfies the following \emph{line property}
        \begin{align*}
            j\prec i\;\mbox{ and }~ i\in L^\mu_\tau ~~\Rightarrow ~~ j\notin L^\mu_\tau
            \qquad \text{ for } i,j\in\Ic\;. 
        \end{align*}
    \end{enumerate}
\end{Definition}
This property ensures that the random set $L^\mu_\tau$ cannot select two particles of the same lineage.
We denote by  $\Sc\Lc^\mu$ the set of stopping lines for the initial condition $\mu$. For $\tau\in\Sc\Lc^\mu$, define the set $D^\mu_\tau$ as
\begin{align}
   \label{eq:def:D_tau}
    D^\mu_\tau := & \Big\{  i\in \Ic~:~\exists j\in\Ic\text{ s.t. }j\prec i \text{ and }j\in L^\mu_\tau \Big\}
    \cup
    \l\{  i\in \Ic~:~i\notin\bigcup_{s\in\R_+}\Vc^\mu_s \r\}
    \;,
\end{align}
which corresponds to the set of particles that are either strict descendants of the line $L^\mu_\tau$ or never appear in the population.
As for stopping times on the real line, the $\sigma$-algebra $\Hc^\mu_\tau$ related to a stopping line $\tau$ is defined as
\begin{align*}
  \Hc^\mu_\tau := &
  \sigma\Big(
    \l\{i\notin D^\mu_\tau\r\} \cap \Hc^i_{\tau_i} \;, ~i\in\Ic
  \Big)\;,
\end{align*}
and its completed $\sigma$-algebra $\Fc^\mu_{\tau}$ defined by
\begin{align*}
  \Fc^\mu_\tau := &
  \sigma\Big(
    \l\{i\notin D^\mu_\tau\r\} \cap \Fc^i_{\tau_i} \;, ~i\in\Ic\Big) \vee \Nc^\P\;,
\end{align*}
with $\Nc^\P$ the collection of all $\P$-negligible sets of $\Fc$.
With respect to the filtration $\F$, we see that $\Fc_s$ corresponds to the filtration generated $\Fc^\mu_{\tau_s}$ with $\tau^s$
\begin{align*}
    \begin{cases}
      \tau^s := s, \qquad &\text{ if }i\in \Vc^\mu_s\eqsp,\\
      \tau^s := S_i, ~~\qquad &\text{ else },
    \end{cases}
\end{align*}
for $\mu\in E$, $i.e.$,
\begin{align*}
    \Fc_s = \sigma\l(\bigcup_{\mu\in E}\Fc^\mu_{\tau_s}\r)\eqsp.
\end{align*}
\subsection{Optimal stopping problem}
To define the optimal stopping problem we introduce a family of reward functions $g_i:~\R^d\rightarrow\R$, $i\in\Ic$, on which we make the following assumptions.
\begin{Assumption} \label{Assumption_H_2} 
\begin{enumerate}[(i)]
  \item The functions $g_i$, for $i\in \Ic$, are positive and bounded uniformly in $i\in \Ic$, $i.e.$, there exists a constant $K_g\geq 1$ such that 
  \begin{align*}
    0\leq g_i(x) \leq K_g\;,\qquad\text{ for }i\in\Ic\eqsp,\eqsp\eqsp x\in\R^d\eqsp.
  \end{align*}
  \item The functions $g_i$, for $i\in \Ic$,  are Lipschitz continuous uniformly in $i\in \Ic$, $i.e.$, there exists a constant $L>0$ such that
  \begin{align*}
    |g_i(x)-g_i(x')| & \leq L|x-x'|\;,\qquad\text{ for }i\in\Ic\eqsp,\eqsp\eqsp x,x'\in\R^d\eqsp.
  \end{align*}
\end{enumerate}
\end{Assumption}

To alleviate notations, we shall denote by 
\begin{align*}
\l(Z^{(i,x)}_t\r)_{t\geq0} &= \l(\sum_{j\in \Vc^{(i,x)}_t}\delta_{(j,X^{(i,x),j}_t)}\r)_{t\geq0}
\end{align*}
the solution to \eqref{eq:semimartingale-decomposition} 
in the case where the initial condition is the measure $\delta_{(i,x)}$ for $(i,x)\in\Ic\times \R^d$. We also use the superscript $(\cdot)^{(i,x)}$ to denote $(\cdot)^{\delta_{(i,x)}}$ in the rest of the paper.

Fix a constant $\gamma>0$. We define the \emph{reward function}, starting from $(i,x)\in\Ic\times \R^d$, as follows 
\begin{align*}
  J_i(x,\tau) 
  &=   \E\left[
    \prod_{j\in L^{(i,x)}_\tau}\rme^{-\gamma \tau_j}
    g_{j}\l(X^{(i,x),j}_{\tau_j} \r) 
  \right]\;,
  \qquad\text{ for }i\in\Ic\eqsp,\eqsp x\in \R^d\eqsp,\eqsp\text{ and } \eqsp\tau\in\Sc\Lc^{{(i,x)}}\eqsp.
\end{align*}

Let $v_i:[0,T]\times \R^d\rightarrow \R$ be the following value function 
\begin{align}
\label{eq:value_fct}
    v_i(x) &= \sup_{\tau\in \Sc\Lc^{{(i,x)}}} J_i(x,\tau)\;,
    \qquad\text{ for }i\in\Ic\eqsp,\eqsp x\in \R^d\eqsp.
\end{align}

This framework can be extended to consider a general initial condition $\mu\in E$, allowing the problem to be formulated with respect to this initial configuration. Specifically, using Theorem \ref{Thm-Branching}, a branching property similar to the one established in \citet{kharroubi2024stochastic,Ocello:Ctrl-superprocesses,Ocello:rel_form_branching} can be leveraged. This approach enables the reduction of the problem to the dynamics of individual particles, providing a more granular perspective while maintaining consistency with the overall population-level behavior.

\begin{Remark}
  \label{rmk:g_i_equal-1}
  It is worth noting that the problem considered in this paper is a generalization of 
  \begin{align*}
      J(x,\tau)
      &=   \E\left[
        \prod_{j\in L^{(\varnothing,x)}_\tau}\rme^{-\gamma \tau_j}g\l(X^{(\varnothing,x),j}_{\tau_j}\r)
      \right]\;,
      \qquad\text{ for }\eqsp x\in \R^d\eqsp,\eqsp\text{ and } \eqsp\tau\in\Sc\Lc^{{(\varnothing,x)}}\eqsp.
  \end{align*}
  when all the $g_i$ are taken to be equal to $g$. In this setting, the system of cost functions (resp. value functions) reduce to a single cost function $J$ (resp. value function $v$). In Remark \ref{rmk:g_i_equal}, we explicitly demonstrate how to go from the generalized problem to this specific case, providing a clear pathway for understanding the relationship between these two formulations. 
  
  Furthermore, by utilizing the Mayer form of the optimal stopping problem---specifically, by appropriately redefining the problem in $\mathbb{R}^{d+1}$---it becomes possible to incorporate a running reward into the framework. In this formulation, the additional dimension accounts for the accumulation of the running reward, allowing the problem to extend beyond a purely terminal reward structure. Additionally, one can naturally consider a weight function at each branching time, as in \citet{henry2012cutting}, which leads to a more general nonlinearity in the PDE system. However, to strike a balance between clarity and generality, we have chosen not to explore this extension within the current work.
\end{Remark}

\section{Regularity and dynamic programming principle}
\label{Section:DPP}

Our objective is to analytically characterize the family of functions $(v_i)_{i\in\Ic}$, subject to the following regularity assumptions on the branching coefficients.
\begin{Assumption}\label{Assumption_H_2BIS}
  \begin{enumerate}[(i)]
      \item The function $\alpha$ is uniformly continuous on $\R^d$.
      \item The functions $p_k$, $k\geq0$, are uniformly continuous on $\R^d$. 
  \end{enumerate}
\end{Assumption}
\noindent We can now establish some fundamental properties of this family.
This result leverages the Lipschitz property of the function $g_i$ and the integrability assumptions on the functions $(p_k)_{k\in\N}$ stated in Assumption \ref{Assumption_H_0BIS}. These conditions play a crucial role in enabling us to effectively estimate the number of particles in the system, ensuring the applicability and robustness of the derived bounds.
\begin{Proposition}\label{Prop:value_fct_properties}
  Suppose that Assumptions \ref{Assumption_H_0}, \ref{Assumption_H_0BIS} and \ref{Assumption_H_2} hold.
  \begin{enumerate}[(i)]
    \item The value functions $v_i$, $i\in\Ic$, are uniformly bounded, $i.e.$,
    \begin{align*}
      v_i(x) & \leq \exp\l(\log(K_g)\eqsp K_g^{\frac{\bar \alpha \bar M}{\gamma}}\r)\;,\qquad \text{ for }i\in\Ic\eqsp,\eqsp x\in\R^d\;.
    \end{align*}
    \item  Under Assumption \ref{Assumption_H_2BIS}, the value functions $v_i$, $i\in\Ic$, are uniformly continuous on $\R^d$, uniformly in $i\in\Ic$.
  \end{enumerate}
\end{Proposition}
\begin{proof}
We suppose $i=\varnothing$. The same argument can be applied for any  $i\in\Ic$.

  (i) We first prove that the $v_i$ are uniformly bounded. Since the function $g_i$, $i\in\Ic$ take values in $[0,K_g]$, we have 
  \begin{align*}
    \rme^{-\gamma \tau_i}g_i(X^{(\varnothing, x),i}_{\tau_i}) < & 1
  \end{align*}
  for any stopping line $\tau\in \Sc\Lc^{{(\varnothing,x)}}$ and any $i\in L^{(\varnothing,x)}_\tau$ such that $S^{{(\varnothing,x)}}_{i-}>T~:=~\frac{\log(K_g)}{\gamma}$. Therefore,
  we can restrict the supremum in the definition of $v_\varnothing(x)$  to the stopping lines $\tau\in \Sc\Lc^{{(\varnothing,x)}}$ such that 
  \begin{align*}
    \tau_i=S^{{(\varnothing,x)}}_i\eqsp,\qquad \text{ if }  S^{{(\varnothing,x)}}_{i-}\geq T\eqsp.
  \end{align*}
  Therefore, applying Proposition \ref{propoestimN}, we have
    \begin{align*}
    v_\varnothing(x) & \leq \E\l[K_g^{\bar N^{{(\varnothing,x)}}_T}\r] \leq K_g^{\rme^{\bar \alpha \bar M T}}\;, \qquad \text{ for } x\in\R^d\eqsp.
    \end{align*}
    
  (ii) We now turn to the uniform continuity property. Fix $x,x'\in\R^d$. For $\tau\in\Sc\Lc^{(\varnothing,x)}$, we define $\tau'=(\tau_i')_{i\in\Ic}$ as 
  \begin{align*}
    \tau_i' &= \tau_i\mathds{1}_{i\in\Vc^{{(\varnothing,x)}}_{\tau_i}\cap \Vc^{{(\varnothing,x')}}_{\tau_i}}+S_i^{{(\varnothing,x')}}\mathds{1}_{i\notin\Vc^{{(\varnothing,x)}}_{\tau_i}\cap \Vc^{{(\varnothing,x')}}_{\tau_i}}
  \end{align*}
    for $i\in\Ic$. We obviously have $\tau'\in\Sc\Lc^{{(\varnothing,x')}}$.
  As in the previous step, we can restrict the supremum in the definition of $v_\varnothing(x)$  to the stopping lines $\tau\in \Sc\Lc^{{(\varnothing,x)}}$ such that 
    \begin{align*}
    \tau_i~=~S_i^{{(\varnothing,x)}}\eqsp, \qquad \text{ if }  S_{i-}^{{(\varnothing,x)}}~\geq~ T~:=~\frac{|\log(K_g)|}{\gamma}\;.
    \end{align*}
    Denote $\Sc\Lc^{{(\varnothing,x)}}_T$ the set of these stopping lines. Note that $\tau'\in\Sc\Lc^{{(\varnothing,x')}}_T$, for $\tau\in\Sc\Lc^{{(\varnothing,x)}}_T$.Moreover, from the definition of $v_\varnothing$, we have
    \begin{align*}
      v_{\varnothing}(x)-v_{\varnothing}(x') & \leq \sup_{\tau\in\Sc\Lc_T^{{(\varnothing,x)}}}J_\varnothing(x,\tau)-J_\varnothing(x',\tau')
      \eqsp.
    \end{align*}

    We now prove that the r.h.s.\ of the previous inequality uniformly converges to $0$ as $|x-x'|$ goes to zero. We have that
    \begin{align*}
      \sup_{\tau\in\Sc\Lc_T^{{(\varnothing,x)}}}J_\varnothing(x,\tau)-J_\varnothing(x',\tau') & \leq \sup_{\tau\in\Sc\Lc_T^{{(\varnothing,x)}}}\E\left[\prod_{j\in L^\mu_\tau
      }\rme^{-\gamma \tau_j}
      g_{j}\l(X^{(\varnothing,x),j}_{\tau_j} \r) - \prod_{j\in L^\mu_{\tau'}
      }\rme^{-\gamma \tau_j'}
      g_{j}\l(X^{(\varnothing,x'),j}_{\tau'_j} \r) 
      \right]\;.
    \end{align*}
    Define now $F_{\delta,T}(x,x')$ as the following set
    \begin{align*}
      F_{\delta,T}(x,x') &=
      \left\{
        \Vc_s^{{(\varnothing,x)}}=\Vc_s^{{(\varnothing,x')}}
        \phantom{ \sup_{i\in \Vc_s^{{(\varnothing,x)}}}}
      \right.
      \\
      &\qquad\qquad\left.\mbox{ and }~ \sup_{i\in \Vc_s^{{(\varnothing,x)}}}\l|X^{(\varnothing,x),i}_s-X^{(\varnothing,x'),i}_s\r|\leq \delta ~~\text{ for } s\in[0,T] \right\}.
    \end{align*}
   From \citet[Lemma 4.2,][]{claisse18-v1}, conditionning on this set, we obtain
   \begin{align*}
      \E\left[\mathds{1}_{F_{\delta,T}(x,x')}\l(\prod_{j\in L^\mu_\tau
      }\rme^{-\gamma \tau_j}
      g_{j}\l(X^{(\varnothing,x),j}_{\tau_j} \r) - \prod_{j\in L^\mu_{\tau'}
      }\rme^{-\gamma \tau_j'}
      g_{j}\l(X^{(\varnothing,x'),j}_{\tau'_j} \r) 
      \r)\right] & \leq  \delta L\E\left[ \bar N^{(\varnothing,x)}_T K_g^{\bar N^{(\varnothing,x)}_T}
      \right]\;.
    \end{align*}
    Fix $\eps>0$. From Proposition \ref{propoestimN}, we can find $\delta_\eps>0$ such that  
    \begin{align}\label{estim-UC1}
      \E\left[\mathds{1}_{F_{\delta_\eps,T}(x,x')}\l(\prod_{j\in L^\mu_\tau
      }\rme^{-\gamma \tau_j}
      g_{j}\l(X^{(\varnothing,x),j}_{\tau_j} \r) - \prod_{j\in L^\mu_{\tau'}
      }\rme^{-\gamma \tau_j'}
      g_{j}\l(X^{(\varnothing,x'),j}_{\tau'_j} \r) 
      \r)\right] & \leq \frac{\eps}{2}\eqsp,
    \end{align}
    for any $x,x'\in\R^d$.
    Moreover, conditionning on the complementary of $F_{\delta_\eps,T}(x,x')$ and using Cauchy--Schwarz's inequality, we have
    \begin{align*}
      &\E\left[\mathds{1}_{F_{\delta_\eps,T}(x,x')^c}\l(\prod_{j\in L^\mu_\tau
      }\rme^{-\gamma \tau_j}
      g_{j}\l(X^{(\varnothing,x),j}_{\tau_j} \r) - \prod_{j\in L^\mu_{\tau'}
      }\rme^{-\gamma \tau_j'}
      g_{j}\l(X^{(\varnothing,x'),j}_{\tau'_j} \r) 
      \r)\right] 
      \\
      & \leq \E\left[\mathds{1}_{F_{\delta_\eps,T}(x,x')^c}K_g^{\bar N_T^{(\varnothing,x)}}\right]
      \\
      & \leq \sqrt{\P(F_{\delta_\eps,T}(x,x')^c)}\sqrt{\E\left[(K_g^2)^{\bar N_T^{(\varnothing,x)}}\right]}\eqsp,
    \end{align*}
    since the functions $g_i$, $i\in\Ic$ are valued in $[0,K_g]$. Combining Proposition \ref{propoestimN} and \citet[Proposition 4.3,][]{claisse18-v1}, we get 
    \begin{align}
    \label{estim-UC2}
      \E\left[\mathds{1}_{F_{\delta_\eps,T}(x,x')^c}\l(\prod_{j\in L^\mu_\tau
      }\rme^{-\gamma \tau_j}
      g_{j}\l(X^{(\varnothing,x),j}_{\tau_j} \r) - \prod_{j\in L^\mu_{\tau'}
      }\rme^{-\gamma \tau_j'}
      g_{j}\l(X^{(\varnothing,x'),j}_{\tau'_j} \r) 
      \r)\right]
      & \leq \frac{\eps}{2}\eqsp,
    \end{align}
  for all $x,x'\in\R^d$ such that $|x-x'|$ is small enough.
  Putting together \eqref{estim-UC1} and \eqref{estim-UC2}, we finally reach the desired uniform continuity property.
\end{proof}

We now derive the \textit{dynamic programming principle} (DPP) for the previously introduced optimization problem.
This result leverages the previously established regularity of the value function and approximate the value functions using $\eps$-optimal stopping lines. This technique is largely employed in the stochastic control literature \citep[see, $e.g.$,][]{claisse18,Bouchard:sto_target:opt_switch}, when minimal regularity conditions of the value function are known. It serves as an alternative approach, circumventing the need for measurable selection results, often complex and intricate.
\begin{Theorem}
\label{thm:DPP}
  Under Assumptions \ref{Assumption_H_0}, \ref{Assumption_H_0BIS}, \ref{Assumption_H_2} and \ref{Assumption_H_2BIS}, we have
  \begin{align}
    \label{eq:DPP}
    \begin{split}
      v_i(x)
      = &
      \sup_{\tau\in\Sc\Lc^{(i,x)}}
      \E\left[
        \prod_{j\in L^{(i,x)}_\theta \setminus D^{(i,x)}_\tau
        }
        \left(\rme^{-\gamma\theta_j}
        v_{j}\left(X^{(i,x),j}_{\theta_j}\right)\right)^{\mathds{1}_{\{ \theta_j\leq \tau_j \}}}
      \right.
      \\
      &\qquad\qquad\qquad
      \left.
        \times \prod_{j\in L_\tau^{(i,x)}\setminus D^{(i,x)}_\theta}
        \left( \rme^{-\gamma\tau_j}
        g_{j}\left(X^{(i,x),j}_{\tau_j}\right)\right)^{\mathds{1}_{\{ \tau_j<\theta_j \}}}
      \right]\;,
    \end{split}
  \end{align}
  for any $(i,x)\in\Ic\times\R^d$ and $\theta\in\Sc\Lc^{(i,x)}$.
\end{Theorem}

The proof of the DPP uses a branching property result for the processes along stopping lines. This property together with the proof is presented in detail in Section \ref{sec:thm:DPP}.

\section{Dynamic programming equation}\label{Section:HJB}

The value function in optimal stopping problems is known to solve an associated obstacle problem---a variational inequality where the value function must satisfy a differential equation in a continuation region and a complementary inequality in a stopping region. We refer to \citet[Theorem 4.5,][]{Touzi:book:opt_ctrl} and \citet[Lemma 5.2.2,][]{pham2009continuous} for the classical case where the dynamics of interest is a diffusion process. This characterization provides a powerful analytical framework for solving optimal stopping problems by reducing them to partial differential or integro-differential equations.

In this section, we aim to extend this framework to our specific value function, demonstrating that it satisfies a similar obstacle problem. To this end, consider the operator $\Lc$ as follows
\begin{align*}
  \Lc: \R^d\times \R\times \R^d\times \S^d\times \R^\N &\to \R\\
  \big(x,r, q,M,(r_\ell)_{\ell\in\N}\big)&\mapsto 
  \frac{1}{2}\text{Tr}\left(\sigma\sigma^\top(x) M\right)+ b(x)^\top q + \alpha(x)\sum_{k\geq0}p_k(x) \prod_{\ell=0}^{k-1} r_{\ell} - (\alpha(x)+ \gamma) r\;,
\end{align*}
with $\S^d$ being the set of symmetric matrices of dimension $d\times d$.
We show in this section that the problem of stopping lines can be characterized by the following PDE
\begin{align}\label{eq:DPE}
\min\left\{-\Lc\left(x,v_i(x), Dv_i(x),D^2v_i(x), \big(v_{i\ell}(x)\big)_{\ell\in\N} \right)~;~
v_i(x)-g_i(x)\right\}=0\;,
\end{align}
for $i\in\Ic$ and $x\in \R^d$. To simplify the notation, we  write $\Lc (i,v)(x)$ for $\Lc\Big(x,v_i(x), Dv_i(x)$, $D^2v_i(x), \big(v_{i\ell}(x)\big)_{\ell\in\N} \Big)$.

\subsection{Verification theorem}
  We now provide a verification theorem in the case where the value functions are regular. This result is a direct consequence of the dynamic programming principle and the dynamic programming equation. It establishes that the value functions are indeed the solution to the optimal stopping problem.
  Recall that, for an initial condition $(i,x)\in\Ic\times\R^d$, and a stopping line $\nu=(\nu_j)_{i\in \Ic}$, we define the set $D^{(i,x)}_{\nu}$ by \eqref{eq:def:D_tau}, with initial measure $\delta_{{(i,x)}}$.

  The proof of this result employs standard arguments, relying on the fact that the semi-linear term diminishes to zero prior to reaching the candidate optimal stopping line.
  \begin{Theorem}
  \label{thm:verification-thm}
    Assume that the value functions $(v_j)_{j\in\Ic}$ satisfy \eqref{eq:DPE}, $i.e.$,
    \begin{enumerate}[(i)]
      \item $v_j\geq g_j$ on $\R^d$ for all $j\in\Ic$;
      \item $v_j$ is twice continuously differentiable with uniformly bounded derivatives on the set $\{v_j>g_j\}$ and
      \begin{align}
      \label{eq:verif:cond}
        \Lc\left(.,v_j(.), Dv_j(.),D^2v_j(.), \big(v_{j\ell}(.)\big)_{\ell\in\N} \right) & = 0 \quad \mbox{ on the set }~ \{ v_j>g_j\}\eqsp,
      \end{align}
      for all $j\in\Ic$.
    \end{enumerate}

    For an initial condition $i\in\Ic$, $x\in\R^n$, define the stopping line $\tau^*$ as
    \begin{align*}
      \tau_{j}^*=\nu_{j}\mathds{1}_{j\notin  D^{(i,x)}_\nu}+S_j^{(i,x)}\mathds{1}_{j\in D^{(i,x)}_\nu}
      \eqsp,
    \end{align*}
    with
    \begin{align*}
      \nu_j & := \inf\l\{ s>S^{(i,x)}_{j-}~:~v_j\l(X^{(i,x),j}_s\r)=g_j\l(X^{(i,x),j}_s\r)\r\}\wedge S^{(i,x)}_j
    \end{align*}
    for $j\in\Ic$. Then, $\tau^*$ is an optimal stopping line, $i.e.$,
    \begin{align}\label{ccl:verif}
      v_{i}(x) & = \E\left[\prod_{i\in L^{(i,x)}_{\tau^*}}e^{-\gamma\tau_j^*}g_j\big(X^{(i,x),i}_{\tau_j^*}\big)\right]\;.
    \end{align}
  \end{Theorem}

\begin{proof}

  To simplify the presentation of the proof, we give the result for an initial label $\varnothing$. The extension to any initial label $i\in\Ic$ is straightforward.
  
  \textbf{Step 1.} We first assume that there exists some $\eps>0$ such that 
  \begin{align}\label{lowBdg}
    g_i(x) & \in [\eps,K_g]\;,\qquad\text{ for } x\in\R^d\eqsp,
  \end{align}
  for any $i\in\Ic$.
  Define the process ${}^\tau Z^{(\varnothing,x)}$ by
  \begin{align*}
    {}^\tau Z^{(\varnothing,x)}_s & =  \sum_{i\in {}^\tau\Vc^{(\varnothing,x)}_s}\delta_{\l(i,{}^{\tau_i}X_{s}^{(\varnothing,x),i}\r)}
    \eqsp,
  \end{align*}
  with
  \begin{align*}
    {}^\tau\Vc^{(\varnothing,x)}_s & = \Vc^{(\varnothing,x)}_s\setminus D^{(\varnothing,x)}_\tau
  \end{align*}
  and
  \begin{align*}
    {}^{\tau_i}X_{s}^{(\varnothing,x),i} & = X_{s\wedge \tau_i}^{(\varnothing,x),i}\;,\qquad\text{ for } i\in {}^\tau\Vc^{(\varnothing,x)}_s\;,s\geq 0\eqsp.
  \end{align*}
  We notice that the process ${}^\tau Z^{(\varnothing,x)}$ satisfies \eqref{eq:semimartingale-decomposition}, with $b^i_s(x)=\mathds{1}_{s\leq\tau_i}b(x)$ (resp. $\sigma^i_s(x)=\mathds{1}_{s\leq\tau_i}b(x)$), for $(s,x)\in\R_+\times\R^d$ in place of $b$ (resp. $\sigma$), and $\tilde Q^i(ds,dz)=\mathds{1}_{i\notin D^{(\varnothing,x)}_{\tau^*}} Q^i(ds,dz)$ in place of $Q^i$, for $i\in\Ic$. Using \eqref{lowBdg}, we apply apply Itô's formula and \eqref{eq:semimartingale-decomposition} to the process $\prod_{i\in {}^\tau\Vc^{(\varnothing,x)}_s}e^{-\gamma(s\wedge \tau_i^*)}v_{i}({}^{\tau_i}X_{s}^{(\varnothing,x),i})$ to get, from 
  \eqref{eq:verif:cond} and the definition of the stopping line $\tau$,
  \begin{align*}
    v_{\varnothing}(x) & = \E\left[\prod_{i\in {}^{\tau^*}\Vc^{(\varnothing,x)}_s}e^{-\gamma(s\wedge \tau_i^*)}v_i\big(X^{(\varnothing,x),i}_{s\wedge\tau_i^*}\big)\right]\;.
  \end{align*}
  
  Proceeding then as in the proof of Proposition \ref{Prop:value_fct_properties} (i), we are able to bound the process $\prod_{i\in {}^{\tau^*}\Vc^{(\varnothing,x)}_s}e^{-\gamma(s\wedge \tau_i^*)}v_i\l(X^{(\varnothing,x),i}_{s\wedge\tau_i^*}\r)$ by an integrable random variable that does not depend on the time $s$. Therefore, sending $s$ to $+\infty$ and applying the dominated convergence theorem, we obtain
  \begin{align*}
    v_{\varnothing}(x) & =
    \E\left[
      \prod_{i\in L^{(\varnothing,x)}_{\tau^*}}e^{-\gamma \tau_i^*}v_i\l(X^{(\varnothing,x),i}_{\tau_i^*}\r)
    \right]\;.
  \end{align*}
  Moreover, since $v_i$ meets $g_i$ at $\tau_i$, we get \eqref{ccl:verif}.
  
  \textbf{Step 2.} We now weaken \eqref{lowBdg} assumming that $g_i\in[0,K_g]$ for all $i\in\Ic$. We next define for $\eps>0$ the value functions $\l(v_i^\eps\r)_{i\in\Ic}$ by \eqref{eq:value_fct} with $g_i+\eps$ in place of $g_i$.
  From Step 1, we get that the stopping line $\tau^{\eps,*}$
  \begin{align*}
    \tau_i^{\eps,*}=\nu_i\mathds{1}_{i\notin  D^{(\varnothing,x)}_{\nu^\eps}}+S_i^{(\varnothing,x)}\mathds{1}_{i\in D^{(\varnothing,x)}_{\nu^\eps}}
    \eqsp,
  \end{align*}
  with
  \begin{align*}
    \nu^{\eps}_i & := \inf\l\{
      s>S^{(\varnothing,x)}_{i-}~:~v^\eps_i(X^{(\varnothing,x),i}_s)=g_i(X^{(\varnothing,x),i}_s)+\eps
    \r\}\wedge S^{(\varnothing,x)}_i
    \eqsp,\qquad\text{ for }i\in\Ic\eqsp,
  \end{align*}
  is optimal for $v^\eps_\varnothing(x)$.
  Using the same arguments as in the proof of Proposition \ref{Prop:value_fct_properties}, we get that $v_i^\eps(x)$ converges uniformly for $(i,x)\in\Ic\times\R^d$ to the function $v_i(x)$. This gives that $\tau_i^{\eps,*}$ converges a.s. to $\tau_i^{*}$ for all $i\in\Ic$. Using again the dominated convergence theorem, we get that $\tau^*$ is optimal for $v_\varnothing(x)$.
\end{proof}
\subsection{Viscosity solutions}

The previously derived PDE reveals a deep interplay with the underlying tree structure, as it couples the value function at node $i$ with those of its direct offspring $i\ell$ for $\ell\geq0$. This coupling is a fundamental aspect of the problem, as it captures how the hierarchical dependencies within the tree influence the evolution of the value function.

We will prove that the value function \eqref{eq:value_fct} is a viscosity solution of \eqref{eq:DPE}. This approach allows us to characterize the value function through a PDE framework, which is particularly suitable given the nature of the problem. It is well-known that solutions to optimal stopping problems often lack strong differentiability properties, making classical solution methods inapplicable. To address this, we adopt the notion of viscosity solutions specifically adapted to our setting, following the framework outlined in \citet[Definition 4.2]{kharroubi2024stochastic}.

\begin{Definition}
Let $u=(u_i)_{i\in\Ic}$ such that $u_i:\R^d\rightarrow\R_+$ is a continuous function for $i\in\Ic$.

  \noindent (i)  $u$ is a \emph{viscosity supersolution} to \eqref{eq:DPE} if, for $(i_0,x_0)\in\Ic\times\R^d$, $\varphi_i\in C^2(\mathbb{R}^{d})$  for $i\in\Ic$,  a constant $C>0$,
   and a function
   $\bar\varphi\in C^0(\mathbb{R}^{d})$
   such that $\varphi_i$ is nonnegative for $i\in\Ic$,
\begin{align}\label{eq:def:visc_sol:sup_phi_bar_phi}
     \varphi_i(x) \leq& C^{|i|}\bar{\varphi}(x),\quad  \text{ for } x\in \mathbb{R}^{d}\;,~i\in \Ic\;,
   \end{align}
  and
  \begin{align*}
  0~=~\left(u_{i_0}- \varphi_{i_0}\right)(x_{0}) = & \min_{\Ic\times \mathbb{R}^{d}}\left(u_{\cdot}-\varphi_\cdot\right)\;,
  \end{align*}
  we have
  \begin{align*}
  \min\Bigg\{-\Lc(i_0,\varphi_\cdot)(x_0)~;~\varphi_{i_0}(x_0)-g_{i_0}(x_0)\Bigg\} \geq  0\;.
  \end{align*}

  \noindent (ii)  $u$ is a \emph{viscosity subsolution} to \eqref{eq:DPE} if, for $(i_0,x_0)\in\Ic\times\R^d$, $\varphi_i\in C^2(\mathbb{R}^{d})$, for $i\in\Ic$,
    a constant $C>0$, and a function
   $\bar\varphi\in C^0(\mathbb{R}^{d})$
  such that $\varphi_i$ is nonnegative, for $i\in\Ic$,
  \eqref{eq:def:visc_sol:sup_phi_bar_phi} is satisfied,
  and 
  \begin{align*}
  0~=~\left(u_{i_0}- \varphi_{i_0}\right)(x_{0}) = & \max_{\Ic\times \mathbb{R}^{d}}\left(u_{\cdot}-\varphi_\cdot\right)\;,
  \end{align*}
  we have
  \begin{align*}
  \min\Bigg\{-\Lc(i_0,\varphi_\cdot)(x_0)~;~\varphi_{i_0}(x_0)-g_{i_0}(x_0)\Bigg\} \leq  0\;.
  \end{align*}
  \noindent (iii)  $u$ is a \emph{viscosity solution} to \eqref{eq:DPE} if it is both a viscosity sub and supersolution to \eqref{eq:DPE}.
\end{Definition}

\begin{Theorem}\label{theorem:result_PDE}
  Under Assumptions \ref{Assumption_H_0}, \ref{Assumption_H_0BIS}, \ref{Assumption_H_2}, and \ref{Assumption_H_2BIS}, the value function $v$ is a viscosity solution to \eqref{eq:DPE}.
\end{Theorem}
The proof of this result is deferred to Section \ref{sec:theorem:result_PDE}.

\subsection{Comparison principle}
We now provide a strong comparison principle for the obstacle problem \eqref{eq:DPE}. To prove this results we extend ideas of the usual comparison principle \citep[see, $e.g.$,][]{pham2018bellman,Touzi:book:opt_ctrl} with the use of \citet{claisse18}. 
\begin{Theorem}\label{Thm:comparison}
  Let $\{u_i\}_{i\in\Ic}$ (resp. $\{v_i\}_{i\in\Ic}$) be a bounded nonnegative continuous viscosity supersolution (resp. subsolution) to \eqref{eq:DPE}. 
  Suppose that there exists a constant $C\geq1$ such that 
  \begin{align}\label{condC1}
    u_i(x)~\leq ~C ~~ \mbox{ and }~~ v_i(x)~\leq~ C\eqsp,
  \end{align}
  for all $(i,x)\in\Ic\times\R^d$ and 
  \begin{align}\label{condC2}
    \gamma & > \bar \alpha\l(
        \bar{M}\frac{C}{C-1}-1
    \r)
    \eqsp.
  \end{align}
  Then, under Assumptions \ref{Assumption_H_0}, \ref{Assumption_H_0BIS}, \ref{Assumption_H_2}, and  \ref{Assumption_H_2BIS}, 
  we have $u_i \leq v_i$ for any $i\in\Ic$ on $\R^d$.
\end{Theorem}
The proof of this result is deferred to Section \ref{sec:theorem:result_PDE}.
As an immediate consequence of Theorems \ref{theorem:result_PDE} and \ref{Thm:comparison}, and Proposition \ref{Prop:value_fct_properties} (i) we have the following characterization of the value function $v$.
 \begin{Corollary}\label{cor:Uniq}
Suppose that 
\begin{align*}
\gamma & > \bar\alpha\l(\bar M\frac{
  \exp\l(\log(K_g)\eqsp K_g^{\frac{\bar \alpha \bar M}{\gamma}}\r)
}{\exp\l(\log(K_g)\eqsp K_g^{\frac{\bar \alpha \bar M}{\gamma}}\r)-1}-1\r).
\end{align*}
Under Assumptions \ref{Assumption_H_0}, \ref{Assumption_H_0BIS}, \ref{Assumption_H_2}, and  \ref{Assumption_H_2BIS},  $v$ is the unique nonnegative  viscosity solution to \eqref{eq:DPE} bounded by $\exp\l(\log(K_g)\eqsp K_g^{\frac{\bar \alpha \bar M}{\gamma}}\r)$. 
\end{Corollary}
We notice that the condition on $\gamma$ in Corollary \ref{cor:Uniq} is satisfied for $\gamma$ large enough. 
\begin{Remark}
\label{rmk:g_i_equal}
  In the case where the functions $g_i$ are the same, that is, $g_i=g$, for $i\in\Ic$, the same property is inherited by the functions $v_i$, $i.e.$, $v_i=v$, for $i\in\Ic$. We indeed have that $(v_{ij})_{j\in\Ic}$ is a viscosity solution to \eqref{eq:DPE}. Therefore, applying Thereom \ref{thm:verification-thm}, in the regular case, or Corollary \ref{cor:Uniq}, we have $(v_{ij})_{j\in\Ic}=(v_{j})_{j\in\Ic}$, entailing $v_i=v_\varnothing=v$, for $i\in\Ic$. In particular, we have that $v$ is the unique viscosity solution to 
  \begin{align*}
    \min\Big\{-\check\Lc\left(x,v(x), Dv(x),D^2v(x) \right)~;~
    v(x)-g(x)\Big\} &= 0\;,
  \end{align*}
  with
  \begin{align*}
    \check\Lc 
    \big(x,r, q,M\big)&:= 
    \frac{1}{2}\text{Tr}\left(\sigma\sigma^\top(x) M\right)+ b(x)^\top q + \alpha(x)\sum_{k\geq0}p_k(x) r^{k} - (\alpha(x)+ \gamma) r\;,
  \end{align*}
  for $(x,r,p,M)\in\R^d\times \R\times \R^d\times \S^d$.
\end{Remark}
\section{Proof of DPP}
\label{sec:thm:DPP}
\subsection{Branching property}
 A key aspect for developing the optimal stopping problem we intend to study is understanding how to condition with respect to the filtration generated by a stopping line. This what we present in the following theorem, generalizing results such as those presented in \citet[Lemma 3.7,][]{claisse18-v1}.
\begin{Theorem}[Branching property]
\label{Thm-Branching}
    Fix an initial condition $\mu \in E$ and two stopping lines $\tau=(\tau_i)_{i\in\Ic}$, $\nu=(\nu_i)_{i\in\Ic}\in \Sc\Lc^\mu$ such that $\nu_i\geq \tau_i$ for all $i\in\Ic$. 
    Then, conditionally to the filtration $\Fc^\mu_\tau$,
    we have that 
    the populations $\sum_{j\in\Vc^\mu_{\nu_j},i\preceq j}\delta_{(j,X^{\mu}_{\nu_j})}$, for $i\in L^\mu_\tau$, 
    satisfy
    \begin{align}
    \label{eq:branching_propST}
    \begin{split}
      &\E\left[
        \prod_{i\in L^\mu_\tau}
        \prod_{k\in L^\mu_\nu,\;k\succeq i}
        f_k\l(
         X^{\mu,k}_{\nu_k}
        \r)
        \middle | \Fc^\mu_{\tau}
      \right](\bar \omega)\\
    &\qquad=
    \prod_{i\in L^\mu_\tau} \E\left[
     \prod_{k\in L^{\bar\mu}_{\nu^{\tau_i(\bar \omega),\bar \omega}-\tau_i(\bar \omega)},\;k\succeq i}
        f_k\l(
         X^{\bar\mu,k }_{\nu_k^{\tau_i(\bar \omega),\bar \omega}-\tau_i(\bar \omega)}
        \r)
      \middle]
      \right|_{
        \bar\mu = \delta_{\l(i,X_{\tau_i}^\mu(\bar \omega)\r)}
      }
    \;,
    \end{split}
    \end{align}
    for $\P$-a.a. $\bar \omega\in\Omega$ and 
    any family $(f_i)_{i\in\Ic}$ of measurable functions from $\R^d$ to $\R_+$.
  \end{Theorem}
\begin{proof}
    This proof is based on \citet[Lemma 3.7,][]{claisse18-v1} that are generalizations to $E$-valued processes of the pseudo-Markov property one provided in \cite{CTT13}.

    Since $\tau=\l(\tau_i\r)_{i\in\Ic}$ (resp. $\nu=\l(\nu_i\r)_{i\in\Ic}$) is a  stopping line, $\tau_i$ (resp. $\nu_i$) is a $(\Fc_t^i)_{t\geq0}$-stopping time for $i\in\Ic$. This means that, for each $i\in\Ic$, there exists an $\l(\Hc^i_t\r)_{t\geq0}$-optional time $\rho_i$ (resp. $\eta_i$) such that
    \begin{align}
    \label{eq:proof:thm-branching:tau_equal_rho}
      \P\l(\tau_i= \rho_i\r)=1\quad\text{ and }\Fc^i_{\tau_i}= \Hc^i_{\rho_i+}\vee \Nc^\P\;,
    \end{align}
    \citep[see, $e.g.$, Chapter II, Theorem 75.3,][]{Rogers-Williams-Vol1}.

    From \eqref{eq:proof:thm-branching:tau_equal_rho}, it is clear that $\rho:= \l(\rho_i\r)_{i\in\Ic}$ is a stopping line. Moreover, since $\Ic$ is a countable set, we also have
    \begin{align*}
       L^\mu_\tau = 
       L^\mu_\rho\;,\quad\P\text{-a.s.}\;,
      \qquad\text{ and }\qquad
      \Fc^\mu_{\tau} = \Hc^\mu_{\rho+}\vee\Nc^\P\;,
    \end{align*}
    with
    \begin{align*}
      \Hc^\mu_{\rho+} := &
    \sigma  \Big(
        \l\{i\notin D^\mu_{\rho}\r\} \cap \Hc^i_{\rho_i+} ~:~i\in\Ic
      \Big)
      \;.
    \end{align*}
    Since $\P$ is a probability measure on the Borel $\sigma$-algebra of a Polish space, there exists $\l(\P_{\bar{\omega}}\r)_{\bar{\omega}\in\Omega}$ a family of regular conditional probabilties (r.c.p.) of $\P$ given $\Hc^\mu_{\rho+}$ \citep[see, $e.g.$, Chapter 5 Section 3.C,][]{Karatzas:Shreve-BMandStochastichCalculus}. It follows from \citet[Lemma 3.2,][]{CTT13} that, for $\P$-a.a. $\bar{\omega}\in\Omega$,
    \begin{align}\label{eq:proof:thm-branching:ICPOmega}
      \begin{split}
        &\P_{\bar{\omega}}\Big(
          D^\mu_\tau= D^\mu_{\tau\l(\bar{\omega}\r)}\l(\bar{\omega}\r)~\text{ and }\\
          &\qquad\qquad
          \tau_i= \tau_i \l(\bar{\omega}\r),~
          X^{\mu,i}_{\tau_i} = X^{\mu,i}_{\tau_i\l(\bar{\omega}\r)}
          \l(\bar{\omega}\r),~
          \nu_j=\nu_j(\bar\omega),~ \text{ for } i \notin D^\mu_\tau,~ j\succeq i
        \Big) = 1\;,
      \end{split}
    \end{align}
    and 
    \begin{align}
    \label{eq:proof:thm-branching:Phi}
      \E\left[
        \prod_{j\in L^\mu_\nu
        }
        f_j\l(
         X^{\mu,j}_{\nu_j}
        \r)
        \middle| \Fc^\mu_\tau
      \right]
        (\bar \omega)
      = &
      \E^{\P_{\bar{\omega}}}\left[
        \prod_{j\in L^\mu_\nu
        }
        f_j\l(
         X^{\mu,j}_{\nu_j}
        \r)
      \right]
      \eqsp,
    \end{align}
    for any Borel functions $f_i:\R\to\R_+$, $i\in\Ic$.
    From \eqref{eq:proof:thm-branching:ICPOmega}, $\nu^{i}(\bar\omega)$ is a stopping line associated to the initial condition $\delta_{\l(i,X^{\mu,i}_{\tau_i(\bar \omega)}(\bar\omega)\r)}$. 

    Define $\hat{Z}^i$ the sub-population generated by the branch $i\in L^\mu_\tau$, selected by the stopping line $\tau$, $i.e.$,
    \begin{align*}
        \hat{Z}^i_{s} := \sum_{j\in\Vc^\mu_{\tau_i+s},i\preceq j}\delta_{\l(j,X^{\mu,j}_{\tau_i+s \wedge (\nu_j-\tau_i)}\r)}
        \eqsp.
    \end{align*}

    Since $Z^\mu$ satisfies the semi-martingale decomposition \eqref{eq:semimartingale-decomposition}, one get by following the same arguments as in the proof of \citet[Lemma 3.7,][]{claisse18-v1} that under $\P_{\bar \omega}$ the process $\hat{Z}^i$ satisfies the semi-martingale decomposition \eqref{eq:semimartingale-decomposition} with the initial condition $\delta_{\l(i,X^{\mu,i}_{\tau_i(\bar \omega)}(\bar\omega)\r)}$ and drift and diffusion coefficients of the branch $j\succeq i$ vanishing after $\nu_j(\bar \omega)-\tau_i(\bar \omega)$, for any $i\in L^\mu_{\tau(\bar \omega)}$.

Using \citet[Corollary 3.6,][]{claisse18-v1}, we get that the law under $\P_{\bar \omega}$ of $\hat{Z}^i_{\cdot}$  coincides with that of the process solution to \eqref{eq:semimartingale-decomposition} with initial condition $\delta_{\l(i,X^{\mu,i}_{\tau_i(\bar \omega)}(\bar\omega)\r)}$
with drift and diffusion coefficients of the branch $j\succeq i$ vanishing after $\nu_j(\bar \omega)-\tau_i(\bar \omega)$ for $\P$-a.a. $\bar \omega\in\Omega$. We deduce that the processes $\hat{Z}^i$, $i\in L^\mu_{\tau}$, are independent under $\P_{\bar \omega}$ for $\P$-a.a. $\bar\omega\in\Omega$ and, from \eqref{eq:proof:thm-branching:Phi}, we get \eqref{eq:branching_propST}.

\end{proof}

This result resembles a branching property similar to the one studied in \citet{kharroubi2024stochastic}. Such properties play a key role in understanding the probabilistic behavior of systems with branching or splitting dynamics, as they allow the translation of the overall population behavior from a macroscopic perspective to an individual-based, microscopic analysis.

\subsection{Proof of Theorem \ref{thm:DPP}}

\begin{proof}[\nopunct]
    Without loss of generality, we can suppose $i=\varnothing$.
    Fix a stopping line $\theta\in\Sc\Lc^{(\varnothing,x)}$ and denote $\bar v(x)$ on the r.h.s. of \eqref{eq:DPP}.
    
    \textbf{Step 1.}
    We first show that $v_\varnothing(x)\leq \bar v(x)$. Fix $\tau\in\Sc\Lc^{(\varnothing,x)}$.
    The idea to follow is to divide the set $L^{(\varnothing,x)}_\tau$ between the particles that have already been stopped when looking at $L^{(\varnothing,x)}_\theta$ and the ones that have not yet been stopped. It is clear that 
    \begin{align*}
    L^{(\varnothing,x)}_\tau=\left(L^{(\varnothing,x)}_\tau\setminus \left( D^{(\varnothing,x)}_\theta\cup L^{(\varnothing,x)}_\theta\right)\right)
    \cup
    \left(L^{(\varnothing,x)}_\tau\cap\left(L^\mu_\theta\cup D^{(\varnothing,x)}_\theta \right) \right)\;.
    \end{align*}
    We divide the branches of $L^{(\varnothing,x)}_\tau$ between the ones belonging to $ D^{(\varnothing,x)}_\theta\cup L^{(\varnothing,x)}_\theta$ and the one in its complementary to get
    \begin{align}
    \label{eq:DPP:1st_splitting}
      \begin{split}
      &J_\varnothing(x,\tau)=\E \left[\prod_{j\in L^{(\varnothing,x)}_\tau}
      \rme^{-\gamma\tau_j}
      g_j\left(X^{(\varnothing,x),j}_{\tau_j} \right)\right] 
      \\
      &=\E \left[
        \prod_{j\in L^{(\varnothing,x)}_\tau\setminus \left( D^{(\varnothing,x)}_\theta\cup L^{(\varnothing,x)}_\theta\right)}
        \rme^{-\gamma\tau_j}g_{j}\left(X^{(\varnothing,x),j}_{\tau_j}\right)
        \prod_{j\in L^{(\varnothing,x)}_\tau\cap\left(L^{(\varnothing,x)}_\theta\cup D^{(\varnothing,x)}_\theta\right)}
        \rme^{-\gamma\tau_j}g_{j}\left(X^{(\varnothing,x),j}_{\tau_j}\right)
      \right]
      \eqsp.
      \end{split}
    \end{align}
    By definition, we have $\tau_j < \theta_j$, for $ j\in L^{(\varnothing,x)}_\tau\setminus \big( D^{(\varnothing,x)}_\theta\cup L^{(\varnothing,x)}_\theta\big)$.
    Therefore, taking the conditional expectation given $\Fc^{(\varnothing,x)}_\theta$, we get
    \begin{align*}
      &J_\varnothing(x,\tau)=
      \E \left[
        \prod_{j\in L^{(\varnothing,x)}_\tau\setminus \left( D^{(\varnothing,x)}_\theta\cup L^{(\varnothing,x)}_\theta\right)}
        \l(\rme^{-\gamma\tau_j}g_{j}\left(X^{(\varnothing,x),j}_{\tau_j}\right)\r)^{\mathds{1}_{\{ \theta_j > \tau_j \}}}
      \right.
      \\
      &\qquad\qquad\qquad\qquad\qquad\qquad\qquad\qquad
      \left.
        \E \left[
            \prod_{j\in L^{(\varnothing,x)}_\tau\cap\left(L^{(\varnothing,x)}_\theta\cup D^{(\varnothing,x)}_\theta\right)}
            \rme^{-\gamma\tau_j}g_{j}\left(X^{(\varnothing,x),j}_{\tau_j}\right)\middle|\Fc_\theta
          \right]
      \right]
      \eqsp.
     \end{align*}
     Splitting the product on $L_\tau^{(\varnothing,x)}\cap L_\theta^{(\varnothing,x)}$ as
      \begin{align*}
        &\prod_{j\in L^{(\varnothing,x)}_\tau\cap L^{(\varnothing,x)}_\theta}
        \rme^{-\gamma\tau_j}
        g_i\left(X^{(\varnothing,x),j}_{\tau_j} \right)
        \\
        &=
        \prod_{j\in L_\tau^{(\varnothing,x)}\cap L_\theta^{(\varnothing,x)}}
        \left(\rme^{-\gamma\tau_j}
        g_{j}\left(X^{(\varnothing,x),j}_{\tau_j}\right)\right)^{\mathds{1}_{\{ \theta_j > \tau_j \}}}
        \prod_{j\in L_\tau^{(\varnothing,x)}\cap L_\theta^{(\varnothing,x)}}
        \left( \rme^{-\gamma\tau_j}
        g_{j}\left(X^{(\varnothing,x),j}_{\tau_j}\right)\right)^{\mathds{1}_{\{ \theta_j\leq \tau_j \}}}
      \end{align*}
      yields
      \begin{align*}
        J_\varnothing(x,\tau)&= 
        \E \left[
          \prod_{j\in L_\tau^{(\varnothing,x)}\setminus D_\theta^{(\varnothing,x)}}
          \left(\rme^{-\gamma\tau_j}g_{j}\left(X^{(\varnothing,x),j}_{\tau_j}\right)\right)^{\mathds{1}_{\{ \theta_j > \tau_j \}}}\right.
          \\
          &\qquad\qquad\left.
            \E\left[
              \prod_{j\in L_\tau^{(\varnothing,x)}\cap L_\theta^{(\varnothing,x)}}\left(\rme^{-\gamma\tau_j}g_{j}\left(X^{(\varnothing,x),j}_{\tau_j}\right)\right)^{\mathds{1}_{\{ \theta_j\leq \tau_j \}}}
            \right.\right.
            \\
            &\qquad\qquad\qquad\qquad\qquad\qquad\qquad\qquad
            \left.\left.\prod_{j\in L_\tau^{(\varnothing,x)}\cap D_\theta^{(\varnothing,x)}}
            \rme^{-\gamma\tau_j}g_{j}\left(X^{(\varnothing,x),j}_{\tau_j}\right)\middle|\Fc_\theta
          \right]
        \right] \;. 
      \end{align*}
      As $D_\theta^{(\varnothing,x)}$ is the set of all particles that belong to the direct descendants of particles in $L_\theta^{(\varnothing,x)}$, we can decompose $L_\tau^{(\varnothing,x)}\cap D_\theta^{(\varnothing,x)}$ as 
      \begin{align*}
        L_\tau^{(\varnothing,x)}\cap D_\theta^{(\varnothing,x)} = &  \bigcup_{j\in L_\theta^{(\varnothing,x)}\setminus D_\tau^{(\varnothing,x)}}\left\{k\in L_\tau^{(\varnothing,x)} : j \prec k \right\}\;.
      \end{align*}
      Combining this with Theorem \ref{Thm-Branching}, we get that \eqref{eq:DPP:1st_splitting} can be rewritten as
      \begin{align*}
        J_\varnothing(x,\tau)&=
        \E \left[
          \prod_{j\in L_\tau^{(\varnothing,x)}\setminus D_\theta^{(\varnothing,x)}}
          \left(\rme^{-\gamma\tau_j}g_{j}\left(X^{(\varnothing,x),j}_{\tau_j}\right)\right)^{\mathds{1}_{\{  \tau_j<\theta_j  \}}}
        \right.
        \\
        &\qquad\qquad\qquad
        \left.\prod_{j\in L_\theta^{(\varnothing,x)} \setminus D_\tau^{(\varnothing,x)}}\left(\rme^{-\gamma\theta_j}J_j\left(X^{(\varnothing,x),j}_{\theta_j},\tau^{j,\theta_j}\right)\right)^{\mathds{1}_{\{\theta_j\leq\tau_j\}}}\right]\;,
      \end{align*}
      with $\tau^{j,\theta_j}$ the  stopping line defined as
      \begin{align*}
        \tau^{j,\theta_j}_{k}(\cdot)
        :=
        \l(
          \l(\tau_{k}-\theta_j\r)\mathds{1}_{\tau_{j}\geq \theta_j}
          +
          \l(+\infty\r)\mathds{1}_{\tau_{k}< \theta_j}\r)(\omega\oplus_{\theta_j(\omega)}\cdot)
          \eqsp,
      \end{align*}
      for a fixed $\omega\in\Omega$ and 
      for $j\in\Ic$, $s\geq0$.
      We see that $\tau^{j,\theta_j}$ corresponds to a stopping line for the initial value $\delta_{\l(j,X^{(\varnothing,x),j}_{\theta_j}\r)}$ from Theorem \ref{Thm-Branching}. Therefore, from the definition of the value function $v$, we get
      \begin{align*}
        &J_\varnothing(x,\tau)\\
        &\leq
        \E \left[\prod_{j\in L_\tau^{(\varnothing,x)}\setminus D_\theta^{(\varnothing,x)}}
        \left(\rme^{-\gamma\tau_j}
        g_{j}\left(X^{(\varnothing,x),j}_{\tau_j}\right)\right)^{\mathds{1}_{\{ \tau_i<\theta_i \}}}\prod_{j\in L_\theta^{(\varnothing,x)}\setminus D_\tau^{(\varnothing,x)} }
        \left(\rme^{-\gamma\theta_j}v_j\left(X^{(\varnothing,x),j}_{\theta_j}\right)\right)^{\mathds{1}_{\{\theta_i\leq\tau_i\}}}\right]\;,
      \end{align*}
      and 
      \begin{align*}
      v_\varnothing(x) \leq & \bar v (x)\;.
      \end{align*}

      \textbf{Step 2.}
      We now turn to the reverse inequality. Fix an open ball $B(x,r)$ centered in $x\in\R^d$ and radius $r >0$.
      Define the  stopping line $\theta^r$ as
      \begin{align*}
        \theta^r_\varnothing
        := &
        \inf\left\{s\geq 0 : X^{(\varnothing,x),\varnothing}_s\notin B(x,r)\right\}\wedge \theta_\varnothing\wedge S_\varnothing^{(\varnothing,x)}\eqsp,
        \\
        \theta^r_{i} 
        := &
        \begin{cases}
          \inf\left\{s\geq S_{i-}^{(\varnothing,x)} : X^{(\varnothing,x),i}_s\notin B(x,r)\right\}\wedge\theta_i \wedge S_i^{(\varnothing,x)}\eqsp, \qquad&\text{ if } \theta^r_{j}= S_j^{(\varnothing,x)} \text{ for any }j\prec i,\\
          S^{(\varnothing,x)}_i\eqsp,
          &\text{  else.}
        \end{cases}
      \end{align*}
      Consider now the following function associated to the stopping line $\theta^r$
      \begin{align*}
        \bar v_r(x):= & \sup_{\tau\in\Sc\Lc}
        \E\left[
          \prod_{i\in L_{\theta^r}^{(\varnothing,x)} \setminus D_\tau
          ^{(\varnothing,x)}}
          \left(\rme^{-\gamma\theta^r_i}
          v_{i}\left(X^{(\varnothing,x),i}_{\theta^r_i}\right)\right)^{\mathds{1}_{\{ \theta^r_i\leq \tau_i \}}}
        \right.
        \\
        &\qquad\qquad\qquad\qquad\left.
          \prod_{i\in L_\tau^{(\varnothing,x)}\setminus D_{\theta^r}^{(\varnothing,x)}}
          \left( \rme^{-\gamma \tau_i}
          g_{i} \left(X^{(\varnothing,x),i}_{\tau_i}\right) \right)^{\mathds{1}_{\{ \tau_i<\theta^r_i \}}}
        \right]\;.
      \end{align*}
      By definition, for a fixed $\eps\in\left(0,1/2\right)$, we can find a stopping line $\tau^\eps$ such that
      \begin{align}
      \label{eq:DPP:difficult_part:1}
        \begin{split}
          \bar v_r(x) \leq ~&
          \E \left[\prod_{i\in L_{\theta^r}^{(\varnothing,x)} \setminus D_{\tau^\eps}^{(\varnothing,x)}}
          \left(\rme^{-\gamma\theta^r_i}
          v_i\l(X^{(\varnothing,x),i}_{\theta^r_i}\r)\right)^{\mathds{1}_{\{ \theta^r_i\leq \tau_i^\eps \}}}
          \right.
        \\
        &\qquad\qquad\qquad\qquad\left.
          \prod_{i\in L_{\tau^\eps}^{(\varnothing,x)}\setminus D_{\theta^r}^{(\varnothing,x)}}\left(\rme^{-\gamma\tau^\eps_i}g_i\l(X^{(\varnothing,x),i}_{ \tau_i^\eps}\r)\right)^{\mathds{1}_{\{ \tau_i^\eps<\theta^r_i \}}}
          \right]+\eps \;.
        \end{split}
      \end{align}
      Consider now a partition $\{B_n\}_n$ of the closure of $B(x,r)$ and a sequence $\{x_n\}_n$ such that $x_n \in B_n$, for $n\geq0$. We can find $\tau^{i,x_n}\in\Sc\Lc^{(i,x_n)}$ such that
      \begin{align}\label{eq:DPP:difficult_part:2}
      v_i(x_n)\leq J_i(x_n, \tau^{i,x_n}) + \eps/3\;,
      \end{align}
      for any $i\in\Ic$.
      Moreover, the proof of in Proposition \ref{Prop:value_fct_properties} shows that we can chose $\tau^{i,x}\in \Sc\Lc^{(i, x)}$ for each $x\in\R^d$ such that $\tau^{i,x}=\tau^{i,x_n}$ for $x=x_n$ and $J_i(\cdot,\tau^{i,\cdot})$ is a uniformly continuous function, uniformly in the index $i\in\Ic$.
      Combining this with the uniform continuity of the value functions $v_i$, uniformly in $i\in\Ic$, we get that the partition $\{B_n\}_n$  and the points $x_n\in B_n$ can be chosen to satisfy 
      \begin{align}\label{eq:DPP:difficult_part:3}
        \max_{i\in\Ic}\left(|v_i(x) - v_i(x_n)| + |J_i(x, \tau^{i,x}) - J_i(x_n,\tau^{i,x_n})|\right) \leq \eps/3\eqsp, \qquad \text{ for }x\in B_n\;.
      \end{align}
      Define the following family of random variables $\bar \tau= (\bar  \tau_i)_{i\in\Ic}$ by
      \begin{align*}
        \bar \tau_{i\ell} & := \tau_{i\ell}^\eps \;,\qquad\text{ for }i\in L^\mu_{\tau^\eps}\setminus D_{\theta^r}\;,~\ell\in\Ic\;,
      \end{align*}
      such that $\tau_i^\eps< \theta^r_i$, and
      \begin{align*}
        \bar \tau_{i\ell} & := ~\theta^r_i + \sum_{n\geq 0} \tau^{i,X^{(\varnothing,x),i}_{\theta_i^r}}_{\ell}
        \1_{B_n}\l(X^{(\varnothing,x),i}_{\theta_i^r}\r) \;,\qquad\text{ for }
        \ell\in\Ic\;,
      \end{align*}
      for $i\in L^{(\varnothing,x)}_{\theta^r}\setminus D^{(\varnothing,x)}_{\tau^\eps}$, such that $\theta^r_i\leq \tau_i^\eps$. Moreover, take $\bar \tau _j=S^{(\varnothing,x)}_j$ if $j\in\Ic$ is not covered by the previous cases. 
    
      Note that $(\bar \tau_i)_{i\in\Ic}\in\Sc\Lc^{(\varnothing,x)}$
      and, from \eqref{eq:DPP:difficult_part:2} and \eqref{eq:DPP:difficult_part:3}, we get
      \begin{align*}
        &\E \left[
          \prod_{i\in L_{\theta^r}^{(\varnothing,x)}\setminus D_{\tau^\eps}^{(\varnothing,x)} }\left(\rme^{-\gamma\theta^r_i}v_i\left(X^{(\varnothing,x),i}_{\theta_i^r}\right)\right)^{\mathds{1}_{\{ \theta^r_i\leq \tau_i^\eps \}}} 
          \prod_{i\in L_{\tau^\eps}^{(\varnothing,x)}\setminus D_{\theta^r}^{(\varnothing,x)}}
          \left(\rme^{-\gamma\tau_i^\eps}g_i\left(X^{(\varnothing,x),i}_{\tau_i^\eps}\right)\right)^{\mathds{1}_{\{ \tau_i^\eps<\theta^r_i \}}}
        \right]
        \\
        &\leq 
        \E \left[
          \prod_{i\in L_{\theta^r}^{(\varnothing,x)}\setminus D_{\bar\tau}^{(\varnothing,x)} }
          \left(\rme^{-\gamma\theta^r_i}\l[
            J_i\left(X^{(\varnothing,x),i}_{\theta^r_i},(\bar \tau_{i\ell}-\theta^r_i
            )_{\ell\in\Ic} \right)+\eps
          \r]
          \right)^{\mathds{1}_{\{ \theta^r_i\leq \bar \tau_i \}}}
        \right.
        \\
        &\qquad\qquad\qquad\qquad\qquad\qquad\qquad\qquad\qquad
        \left.
          \prod_{i\in L_{\bar\tau}^{(\varnothing,x)}\setminus D_{\theta^r}^{(\varnothing,x)}}
          \left(\rme^{-\gamma\bar\tau_i}
          g_i\left(X^{(\varnothing,x),i}_{\bar \tau_i}\right)\right)^{\mathds{1}_{\{ \bar \tau_i<\theta^r_i \}}}
        \right]   \;.
      \end{align*}
      Following the same argument as in the proof of Proposition \ref{Prop:value_fct_properties} (i) and using \citet[Lemma 4.2 ,][]{claisse18-v1}, we get that there exists a constant $C>0$ such that 
      \begin{align*}
        &\prod_{i\in L_{\theta^r}^{(\varnothing,x)}\setminus D_{\bar\tau}^{(\varnothing,x)} }
        \left(
          \rme^{-\gamma\theta^r_i}\l[J_i\left(X^{(\varnothing,x),i}_{\theta_i^r},(\bar \tau_{i\ell}-\theta^r_i
          )_{\ell\in\Ic} \right)+\eps\r]
        \right)^{\mathds{1}_{\{ \theta^r_i\leq \bar \tau_i \}}}
        \\
        &\qquad-\prod_{i\in L_{\theta^r}^{(\varnothing,x)}\setminus D_{\bar\tau}^{(\varnothing,x)} }
        \left(\rme^{-\gamma\theta^r_i}J_i\left(X^{(\varnothing,x),i}_{\theta_i^r},(\bar \tau_{i\ell}-\theta^r_i
        )_{\ell\in\Ic} \right)
        \right)^{\mathds{1}_{\{ \theta^r_i\leq \bar \tau_i \}}}  \leq  
        ~\bar N_T C^{\bar N_T}L\eps
        \eqsp,
      \end{align*}
      with $T:=\log(K)/\gamma$.
      Therefore, we obtain
      \begin{align*}
        &\E \left[
          \prod_{i\in L_{\theta^r}^{(\varnothing,x)}\setminus D_{\tau^\eps}^{(\varnothing,x)} }\left(\rme^{-\gamma\theta^r_i}v_i\left(X^{(\varnothing,x),i}_{\theta_i^r}\right)\right)^{\mathds{1}_{\{ \theta^r_i\leq \tau_i^\eps \}}} 
          \prod_{i\in L_{\tau^\eps}^{(\varnothing,x)}\setminus D_{\theta^r}^{(\varnothing,x)}}
          \left(\rme^{-\gamma\tau_i^\eps}g_i\left(X^{(\varnothing,x),i}_{\tau_i^\eps}\right)\right)^{\mathds{1}_{\{ \tau_i^\eps<\theta^r_i \}}}
        \right]
        \\
        & \leq  
        \E \left[
          \prod_{i\in L_{\theta^r}^{(\varnothing,x)}\setminus D_{\bar\tau}^{(\varnothing,x)} }
          \left(\rme^{-\gamma\theta^r_i}J_i\left(X^{(\varnothing,x),i}_{\theta_i^r},(\bar \tau_{i\ell}-\theta^r_i
          )_{\ell\in\Ic} \right)
          \right)^{\mathds{1}_{\{ \theta^r_i\leq \bar \tau_i \}}}
        \right.
        \\
        &\qquad\qquad\qquad\qquad\qquad
        \left.
          \prod_{i\in L_{\bar\tau}^{(\varnothing,x)}\setminus D_{\theta^r}^{(\varnothing,x)}}
          \left(\rme^{-\gamma\bar\tau_i}
          g_i\left(X^{(\varnothing,x),i}_{\bar \tau_i}\right)\right)^{\mathds{1}_{\{ \bar \tau_i<\theta^r_i \}}}
        \right]
        +\eps L\E\l[\bar N_T C^{\bar N_T}\r]\;. 
      \end{align*}
      Using Theorem \ref{Thm-Branching}, we get
      \begin{align*}
        &\E \left[
          \prod_{i\in L_{\theta^r}^{(\varnothing,x)}\setminus D_{\bar\tau}^{(\varnothing,x)} }
          \left(\rme^{-\gamma\theta^r_i}J_i\left(X^{(\varnothing,x),i}_{\theta_i^r},(\bar \tau_{i\ell}-\theta^r_i
          )_{\ell\in\Ic} \right)
          \right)^{\mathds{1}_{\{ \theta^r_i\leq \bar \tau_i \}}}
        \right.
        \\
        &\qquad\qquad\qquad\qquad\quad
        \left.
          \prod_{i\in L_{\bar\tau}^{(\varnothing,x)}\setminus D_{\theta^r}^{(\varnothing,x)}}
          \left(\rme^{-\gamma\bar\tau_i}
          g_i\left(X^{(\varnothing,x),i}_{\bar \tau_i}\right)\right)^{\mathds{1}_{\{ \bar \tau_i<\theta^r_i \}}}
        \right]
        =
        \E \left[\prod_{i\in L_{\bar \tau}^{(\varnothing,x)}}\rme^{-\gamma\bar\tau_i}g_i\left(X^{(\varnothing,x),i}_{\bar \tau_i}\right)\right]  \;.
      \end{align*}
      This yields that conclusion applying Proposition \ref{propoestimN} and sending $\eps$ to zero and $r$ to $+\infty$.
    \end{proof}

\section{Proof of viscosity solution properties}

\label{sec:theorem:result_PDE}

To prove the viscosity properties of the problem under consideration, we first understand how the tree structure influences the evolution of the value function. This will yield to the characterization of the value function. We begin by examining the following preliminary proposition. This result shows that, under Assumption \ref{Assumption_H_0BIS}, the operator $\Lc$ is well-defined for bounded sequences of functions. Additionally, it provides further insights into the behavior of the operator and its connection to the tree structure.

\begin{Proposition}\label{prop:results:p_k}
  (i) Under Assumption \ref{Assumption_H_0BIS}, the series  $y\in\R^d\mapsto\sum_{k\geq0}p_k(x) |y|^k$ has infinite radius of convergence for any $x\in\R^d$.

  (ii) Under Assumptions \ref{Assumption_H_0}, \ref{Assumption_H_0BIS}, and \ref{Assumption_H_2BIS}, we have
  \begin{align*}
    \sum_{k\geq0} |p_k(y)-p_k(x)|R^k \xrightarrow[|x-y|\rightarrow0]{} 0\eqsp,\qquad\text{ for }R>0\eqsp.
  \end{align*}
\end{Proposition}
\begin{proof}
  
  (i) Fix some $R>1$. We then have
  \begin{align*}
  \sum_{k\geq0}p_k(x) R^k &= \sum_{k\geq0}p_k(x) \rme^{k\log(R)}\\
  &= \sum_{k\geq0}p_k(x) \sum_{\ell\geq0}\frac{k^\ell\log(R)^\ell}{\ell !}\\
  &= \sum_{\ell\geq0}\frac{\log(R)^\ell}{\ell !}\sum_{k\geq0}p_k(x)k^\ell\\
    & \leq \sum_{\ell\geq0}\frac{\log(R)^\ell M_\ell}{\ell !}\\
      & \leq C\sqrt{R}
  \end{align*}
  where the constant $C$ is such that $C\geq \sup_{\ell}2^\ell M_\ell$.

  (ii) Note that it is sufficient to prove the result for $R>1$. Suppose then that $R>1$ and let $\rho$ be the probability measure on $\N$ defined by
  \begin{align*}
    \rho &= \frac{R-1}{R}\sum_{k\geq0}\frac{1}{R^k}\delta_{\{k\}}\;.
  \end{align*}
  Therefore,
  \begin{align*}
    \sum_{k\geq0} |p_k(y)-p_k(x)|R^k &= \int_\N f_{x,y}(k)d\rho(k)
  \end{align*}
  where $f_{x,y}(k)=\frac{R}{R-1}|p_k(y)-p_k(x)|R^{2k}$, for $k\in\N$ and $x,y\in\R^d$. Assumption \ref{Assumption_H_2BIS} yields that
  \begin{align*}
    f_{x,y}(k) & \xrightarrow[|x-y|\rightarrow0]{} 0\;.
  \end{align*}
  Moreover,
  \begin{align*}
    \int_\N |f_{x,y}(k)|^2d\mu(k) &= \frac{R}{R-1}\sum_{k\geq0} |p_k(y)-p_k(x)|^2R^{3k}\\
    & \leq \frac{R}{R-1}\Big(\sum_{k\geq0} p_k(y)R^{3k}+\sum_{k\geq0}p_k(x)R^{3k}\Big)
    \eqsp.
  \end{align*}
  Proceeding as in the first step, we get 
  \begin{align*}
    \int_\N |f_{x,y}(k)|^2d\mu(k) &= 2C\sqrt{R}
  \end{align*}
  where the constant $C$ is such that $C\geq \sup_{\ell}2^\ell M_\ell$. Therefore, the family $(f_{x,y})_{x,y\in\R^d}$ is uniformly integrable and we get the result from the dominated convergence theorem.
\end{proof}

\subsection{Proof of Theorem \ref{theorem:result_PDE}}

\begin{proof}[\nopunct]
  \textbf{Step 1: supersolution property.}
  Fix $(i_0,x_{0})\in\Ic\times \mathbb{R}^{d}$ and let $\bar\varphi\in C^0(\R^d)$, $C>0$, and $\varphi_i\in C^{2}( \mathbb{R}^{d})$ satisfying \eqref{eq:def:visc_sol:sup_phi_bar_phi}, for $i\in \Ic$,
  and
  \begin{align}\label{condMin}
  0=\left(v_{i_0}- \varphi_{i_0}\right)(x_{0})=\min_{(i,x)\in \Ic\times \mathbb{R}^{d}}\left(v_{i}-\varphi_{i}\right)(x)\;.
  \end{align}
  Without loss of generality, we can assume $i_0=\varnothing$ and  the minimum in \eqref{condMin} to be strict in $x$. To simplify notation and avoid unnecessary complexity, we omit the superscript $\{\cdot\}^{(\varnothing,x_{0})}$.

  Consider, first, the following (trivial) stopping line $\tau^{\text{triv}}$
  \begin{align*}
    \tau^{\text{triv}}_\varnothing:= 0, ~\text{ and }~ \tau^{\text{triv}}_j:=S_j,~ \text{ for }j\in \Ic\backslash\{\varnothing\}\;.
  \end{align*}
  Combining it with \eqref{condMin}, we get the inequality $v_{\varnothing}(x_{0})=\varphi_{\varnothing}(x_{0})\geq g_{\varnothing}(x_0)$.

  Fix $h>0$. Define the stopping time $\bar\theta^h$ as follows
  \begin{align*}
    \bar\theta^h := \inf\left\{t>0: X^{i}_t\notin  B(x_0,1)\eqsp,\eqsp i\in \Vc_t\right\}\wedge h\;.
  \end{align*}
  We associated with the stopping time $\bar\theta^h$ the stopping line
  and define the following stopping line $\theta^h$ that gives either the exit time $\bar\theta^h$ or at the branching time $S_\varnothing$, in the case it arrives before this exit time, $i.e.$,
  \begin{align*}
    \theta^h_\varnothing
    := ~&
    \bar\theta^h\wedge S_\varnothing
    \eqsp,
    \\
    \theta^h_\ell
    := ~&
    \begin{cases}
      S_\ell\eqsp, \quad &\text{ if }\bar\theta^h< S_\varnothing
      \eqsp,
      \\
      S_\varnothing\eqsp,  ~~&\text{ else}\eqsp,
    \end{cases}
    \qquad \text{ for }\ell\in \N\eqsp,
    \\
    \theta^h_j:= ~&S_j\eqsp,
    \qquad\qquad\qquad\qquad\qquad~ \text{ for }j\in \Ic\backslash\left(\{\varnothing\}\cup \N\right)\;.
  \end{align*}
  Therefore, from \eqref{eq:DPP} w.r.t.\ the stopping lines $\theta=\theta^h$ and $\tau=\theta^h$,
  we have
  \begin{align*}
    v_{\varnothing}(x_0)  \geq ~& \E\left[
      \rme^{-\gamma {\bar\theta^h}}
      v_{\varnothing}\left(X^{\varnothing}_{\bar\theta^h}\right) \1_{\bar\theta^h< S_\varnothing } +
      \prod_{\ell=0}^{\l|\Vc_{S_\varnothing}\r|-1}
      \left(\rme^{-\gamma {S_\varnothing}} v_{\ell}\left(X^{\ell}_{S_\varnothing}\right)\right) \1_{\bar\theta^h\geq S_\varnothing}
    \right]\;.
  \end{align*}
  Since we are considering a local branching, $i.e.$, $X^{\ell}_{S_\varnothing}=X^{\varnothing}_{S_\varnothing-}$, we obtain
  \begin{align*}
    v_{\varnothing}(x_0) 
    \geq ~&
    \E\left[
      \rme^{-\gamma {\bar\theta^h}}
      v_{\varnothing}\left(X^{\varnothing}_{\bar\theta^h}\right) \1_{\bar\theta^h< S_\varnothing } +
      \prod_{\ell=0}^{\l|\Vc_{S_\varnothing}\r|-1}
      \left(\rme^{-\gamma {S_\varnothing}} v_{\ell}\left(X^{\varnothing}_{S_\varnothing-}\right)\right) \1_{\bar\theta^h\geq S_\varnothing}
    \right]
    \\
    \geq ~&
    \E\left[
      \rme^{-\gamma {\bar\theta^h}}
      \varphi_{\varnothing}\left(X^{\varnothing}_{\bar\theta^h}\right) \1_{\bar\theta^h< S_\varnothing } +
      \prod_{\ell=0}^{\l|\Vc_{S_\varnothing}\r|-1}
      \left(\rme^{-\gamma {S_\varnothing}} \varphi_{\ell}\left(X^{\varnothing}_{S_\varnothing-}\right)\right) \1_{\bar\theta^h\geq S_\varnothing}
    \right]
    \;.
  \end{align*}
  where in the last inequality we used \eqref{condMin}, as the functions $v_{j}$ and $\varphi_{j}$ are positive for $j\in\Ic$.
  From the dynamics \eqref{eq:semimartingale-decomposition} of the process $Z$, we get
  \begin{align*}
    &\varphi_{\varnothing}(x_0)
    \\
    &\geq~
    \E\left[
      \rme^{-\gamma {\bar\theta^h}}
      \varphi_{\varnothing}\left(X^{\varnothing}_{\bar\theta^h}\right) \1_{\bar\theta^h< S_\varnothing }
      +
      \sum_{k\geq 1}(\alpha p_k)\l(X^{\varnothing}_{S_\varnothing-}\r)\prod_{\ell=0}^{k-1}
        \left(\rme^{-\gamma {S_\varnothing}} \varphi_{\ell}\left(X^{\varnothing}_{S_\varnothing-}\right)\right) \1_{\bar\theta^h\geq S_\varnothing}
    \right]
    \\
    &=~
    \E\left[
      \rme^{-\gamma {\bar\theta^h\wedge  S_\varnothing }}
      \varphi_{\varnothing}\left(X^{\varnothing}_{\bar\theta^h \wedge  S_\varnothing-}\right)
      +
      \sum_{k\geq 1}(\alpha p_k)\l(X^{\varnothing}_{S_\varnothing-}\r)\prod_{\ell=0}^{k-1}
      \left(\rme^{-\gamma {S_\varnothing}} \big(\varphi_{\ell}-\varphi_{\varnothing}\big)\left(X^{\varnothing}_{S_\varnothing-}\right)\right) \1_{\bar\theta^h\geq S_\varnothing}
    \right]
    \\
    &=~
    \E\left[
      \rme^{-\gamma {\bar\theta^h\wedge  S_\varnothing }}
      \varphi_{\varnothing}\left(X^{\varnothing}_{\bar\theta^h \wedge  S_\varnothing-}\right)
      +
      \sum_{k\geq 1}\int_0^{\bar\theta^h}(\alpha p_k)\l(X^{\varnothing}_s\r)\prod_{\ell=0}^{k-1}
      \rme^{-\gamma s} \big(\varphi_{\ell}-\varphi_{\varnothing}\big)\left(X^{\varnothing}_{s}\right)\rmd s
    \right]\;.
\end{align*}
Applying Itô's formula, we get
\begin{align*}
  0\geq~ &
  \E\left[
    \int_0^{\bar\theta^h\wedge S_\varnothing } \rme^{-\gamma {s}}\left(
    \frac{1}{2}\text{Tr}\left(\sigma\sigma^\top D^2\varphi_{\varnothing}\right) + \left(b^\top D \varphi_{\varnothing}\right) - \gamma \varphi_{\varnothing}
    \right) \left(X^{\varnothing}_{s}\right)\rmd s
  \right.\\
  &
  \left.
    \qquad\qquad\qquad\qquad\qquad+
    \sum_{k\geq 1}\int_0^{\bar\theta^h}\alpha p_k\l(X^{\varnothing}_s\r)\prod_{\ell=0}^{k-1}
    \rme^{-\gamma {s}} \big(\varphi_{\ell}-\varphi_{\varnothing}\big)\left(X^{\varnothing}_{s}\right)\rmd s
  \right]\;.
\end{align*}
Dividing by $h>0$ both sides of previous inequality, from the mean value theorem and the dominated convergence theorem, we have that $-\Lc\left(i_0,\varphi_{\cdot}\right)(x_0)\geq 0$.

  \textbf{Step 2: subsolution property.}
  Fix $(i_0,x_{0})\in\Ic\times \mathbb{R}^{d}$ and let $\varphi\in C^0(\R^d)$, $C>0$, and $\varphi_i\in C^{2}( \mathbb{R}^{d})$ satisfying \eqref{eq:def:visc_sol:sup_phi_bar_phi}, for $i\in \Ic$, and
  \begin{align}\label{condMax}
  0=\left(v_{i_0}- \varphi_{i_0}\right)(x_{0})=\max_{(i,x)\in \Ic\times \mathbb{R}^{d}}\left(v_{i}-\varphi_{i}\right)(x)~.
  \end{align}
  Without loss of generality, we suppose that $i_0=\varnothing$ and take the maximum to be strict in $x$ and that 
  \begin{align}
  \label{condtestfunci0ell}
    \max_{(\ell,x)\in \N\times \mathbb{R}^{d}}\left(v_{\ell}-\varphi_{\ell}\right)(x)= & -\delta<0
    \;.
  \end{align}
  As for the previous step, we omit the superscript $\{\cdot\}^{(\varnothing,x_{0})}$.

  We argue by contradiction and assume that
  \begin{align*}
    2\eta
    :=
    \min\Big\{-\Lc(\varnothing,\varphi_\cdot)(x_0)~;~\varphi_{\varnothing}(x_0)-g_{\varnothing}(x_0)\Big\}>0\;.
  \end{align*}
  From the continuity of the functions involved in the previous inequality, we may find $\eps>0$ such that
  \begin{align}
  \label{subsol:inside_L_geq_eta}
    -\Lc(\varnothing,\rme^{-\gamma s}(\varphi_\cdot-y))(x) &>~\eta \eqsp,
  \\
  \label{subsol:jump_geq_eta}
    \left(\varphi_{\varnothing} - g_{\varnothing}\right)(x)&>~ \eta\;,
  \end{align}
  for all $s,y\in[0,\eps)$ and $x\in  B(x_0,\eps)$.
  Note that
  \begin{align}\label{subsol:parabolic_max_zeta}
    - \zeta = & \max_{\partial B_\eps(x_0)}(v_{\varnothing}- \varphi_{\varnothing})(x) <0\;,
  \end{align}
  as $x_0$ is a strict maximizer,
  where $\partial B_\varepsilon(x_0)$ denotes the boundary of $B_\varepsilon(x_0)$.
  
  We now show that \eqref{subsol:inside_L_geq_eta}, \eqref{subsol:jump_geq_eta}, and \eqref{subsol:parabolic_max_zeta} lead to a contradiction of \eqref{eq:DPP}.
  Define the stopping time $\bar \theta^\eps$ as
  \begin{align*}
    \bar\theta^\eps := \inf\Big\{t>0: \left(t,X^{i}_{t}\right)\notin[0,\eps)\times B_\eps(x_0)\eqsp,\eqsp i\in \Vc_t\Big\}\;.
  \end{align*}
  As for the supersolution property, we consider the stopping line $\theta^\eps$ as
  \begin{align*}
    \theta^\eps_\varnothing
    := ~&
    \bar\theta^\eps\wedge S_\varnothing
    \eqsp,
    \\
    \theta^\eps_\ell
    := ~&
    \begin{cases}
      S_\ell\eqsp, \quad &\text{ if }\bar\theta^\eps< S_\varnothing
      \eqsp,
      \\
      S_\varnothing\eqsp,  ~~&\text{ else}\eqsp,
    \end{cases}
    \qquad \text{ for }\ell\in \N\eqsp,
    \\
    \theta^\eps_j:= ~&S_j\eqsp,
    \qquad\qquad\qquad\qquad\qquad~ \text{ for }j\in \Ic\backslash\left(\{\varnothing\}\cup \N\right)\;.
  \end{align*}
  From \eqref{condtestfunci0ell} and \eqref{subsol:parabolic_max_zeta}, we get
  \begin{align*}
    &v_{\varnothing}(x_0)
    -
    \E\left[
      \prod_{j\in L_{\theta^\eps}{\setminus D_\tau}  }
      \left(\rme^{-\gamma\theta^\eps_j}
      v_{j}\left(X^{j}_{\theta^\eps_j}\right)\right)^{\mathds{1}_{\l\{ \theta^\eps_j\leq \tau_j \r\}}}
      \prod_{j\in L_{\tau}\setminus D_{\theta^\eps}}
      \left( \rme^{-\gamma\tau_j}
      g_{j}\left(X^{j}_{\tau_j}\right)\right)^{\mathds{1}_{\l\{ \tau_j<\theta^\eps_j \r\}}}
    \right]
    \\
    & =
    \varphi_{\varnothing}(x_0) - 
    \E\left[
      \prod_{j\in L_{\theta^\eps} {\setminus D_\tau} }
    \left(\rme^{-\gamma\theta^\eps_j}
    v_{j}\left(X^{j}_{\theta^\eps_j}\right)\right)^{\mathds{1}_{\l\{ \theta^\eps_j\leq \tau_j \r\}}}
    \prod_{j\in L_{\tau}\setminus D_{\theta^\eps}}
    \left( \rme^{-\gamma\tau_j}
    g_{j}\left(X^{j}_{\tau_j}\right)\right)^{\mathds{1}_{\l\{ \tau_j<\theta^\eps_j \r\}}}\right]
    \\
    &=
    \varphi_{\varnothing}(x_0) - 
    \E\left[
      \mathds{1}_{\l\{ \bar \theta^\eps< S_\varnothing\r\}}  
      \left(\rme^{-\gamma \bar\theta^\eps}
      v_{\varnothing}\left(X^{\varnothing}_{\bar\theta^\eps}\right)\right)^{\mathds{1}_{\l\{\bar \theta^\eps\leq \tau_\varnothing \r\}}}
      \left( \rme^{-\gamma \tau_\varnothing}
      g_{\varnothing}\left(X^{\varnothing}_{\tau_\varnothing}\right)\right)^{\mathds{1}_{\l\{ \tau_\varnothing<\theta^\eps_\varnothing \r\}}}
    \right]
    \\
    &\quad
    - \E\left[
    \mathds{1}_{\l\{ \bar \theta^\eps\geq S_\varnothing\r\}}
    \left(
      \mathds{1}_{\l\{ \tau_\varnothing\geq S_\varnothing\r\}}
      \left(
        \prod_{\ell=0 }^{\l|\Vc_{S_\varnothing}\r|-1}
        \rme^{-\gamma S_\varnothing }
        v_{\ell}\left(X^{\varnothing}_{S_\varnothing}\right)
      \right)
      +
      \mathds{1}_{\l\{ \tau_\varnothing< S_\varnothing\r\}}
      \rme^{-\gamma \tau_\varnothing}
      g_{\varnothing}\left(X^{\varnothing}_{\tau_\varnothing}\right)
    \right)\right]
    \\
    &\geq ~
    \varphi_{\varnothing}(x_0) - 
    \E\left[
      \mathds{1}_{\l\{ \bar \theta^\eps< S_\varnothing\r\}}  
      \left(\rme^{-\gamma \bar\theta^\eps}
      \left(\varphi_{\varnothing}\left(X^{\varnothing}_{\bar\theta^\eps}\right)-\zeta\right){\mathds{1}_{\l\{\bar \theta^\eps\leq \tau_\varnothing \r\}}}
      +
      \rme^{-\gamma \tau_\varnothing}
      g_{\varnothing}\left(X^{\varnothing}_{\tau_\varnothing}\right){\mathds{1}_{\l\{ \tau_\varnothing<\theta^\eps_\varnothing \r\}}}\right)
    \right]
    \\
    &\quad
    - \E\left[
      \mathds{1}_{\l\{ \bar \theta^\eps\geq S_\varnothing\r\}}
      \phantom{\prod_{\ell=0 }^{\l|\Vc_{S_\varnothing}\r|-1}}
      \right.
      \\
      &\qquad\qquad\left.
      \left(
        \mathds{1}_{\l\{ \tau_\varnothing\geq S_\varnothing\r\}}
        \left(\prod_{\ell=0 }^{\l|\Vc_{S_\varnothing}\r|-1}
        \rme^{-\gamma S_\varnothing}\left(
        \varphi_{\ell}\left(X^{\varnothing}_{S_\varnothing}\right)-\delta\right)\right)
        +
        \mathds{1}_{\l\{ \tau_\varnothing< S_\varnothing\r\}}
        \rme^{-\gamma \tau_\varnothing}
        g_{\varnothing}\left(X^x_\varnothing(\tau_\varnothing)\right)
      \right)
    \right]
    \\
    &\geq
    \varphi_{\varnothing}(x_0) - 
    \E\left[
    \mathds{1}_{\l\{ \bar \theta^\eps< S_\varnothing\r\}}  
    \rme^{-\gamma \bar\theta^\eps}
    \left(\varphi_{\varnothing}\left(X^{\varnothing}_{\bar\theta^\eps}\right)-\zeta\right){\mathds{1}_{\l\{\bar \theta^\eps\leq \tau_\varnothing \r\}}}
    \right] 
    \\&
    \qquad\qquad\qquad
    - \E\left[
    \mathds{1}_{\l\{ \bar \theta^\eps\geq S_\varnothing\r\}}
    \mathds{1}_{\l\{ \tau_\varnothing\geq S_\varnothing\r\}}
    \prod_{\ell=0 }^{\l|\Vc_{S_\varnothing}\r|-1}
    \rme^{-\gamma S_\varnothing}
    \left(\varphi_{\ell}\left(X^{\varnothing}_{S_\varnothing}\right)-\delta\right)
    \right]
    \\
    &\qquad\qquad\qquad
      -\E\left[
      \mathds{1}_{\{ \tau_\varnothing< \bar \theta^\eps\wedge S_\varnothing\}}
      \rme^{-\gamma \tau_\varnothing}
      g_{\varnothing}\left(X^{\varnothing}_{\tau_\varnothing}\right)\right] \;,
  \end{align*}
for any $\tau\in\Sc\Lc$.
Using \eqref{subsol:jump_geq_eta}, we get
\begin{align*}
  &v_{\varnothing}(x_0)
    -
    \E\left[
      \prod_{j\in L_{\theta^\eps}{\setminus D_\tau}  }
      \left(\rme^{-\gamma\theta^\eps_j}
      v_{j}\left(X^{j}_{\theta^\eps_j}\right)\right)^{\mathds{1}_{\l\{ \theta^\eps_j\leq \tau_j \r\}}}
      \prod_{j\in L_{\tau}\setminus D_{\theta^\eps}}
      \left( \rme^{-\gamma\tau_j}
      g_{j}\left(X^{j}_{\tau_j}\right)\right)^{\mathds{1}_{\l\{ \tau_j<\theta^\eps_j \r\}}}
    \right]
    \\
    &\geq ~
    \varphi_{\varnothing}(x) - 
    \E\left[
      \mathds{1}_{\l\{ \bar \theta^\eps< S_\varnothing\r\}}  
      \rme^{-\gamma \bar\theta^\eps}
      \left(\varphi_{\varnothing}\left(X^{\varnothing}_{\bar\theta^\eps}\right)-\zeta\right){\mathds{1}_{\l\{\bar \theta^\eps\leq \tau_\varnothing \r\}}}
    \right]
    \\&
    \qquad\qquad\qquad
    - \E\left[
      \mathds{1}_{\l\{ \bar \theta^\eps\geq S_\varnothing\r\}}
      \mathds{1}_{\l\{ \tau_\varnothing\geq S_\varnothing\r\}}
      \prod_{\ell=0 }^{\l|\Vc_{S_\varnothing}\r|-1}
      \rme^{-\gamma S_\varnothing}
      \Big(\varphi_{\ell}\left(X^{\varnothing}_{S_\varnothing}\right)-\delta\Big)
    \right]
    \\&
    \qquad\qquad\qquad\qquad\qquad\qquad\qquad\qquad
    -\E\left[
    \mathds{1}_{\{ \tau_\varnothing< \bar \theta^\eps\wedge S_\varnothing\}}
    \rme^{-\gamma \tau_\varnothing}\left(
    \varphi_{\varnothing}\left(X^{\varnothing}_{\tau_\varnothing}\right)-\eta\right)\right]
    \\
    &\geq ~ 
    \varphi_{\varnothing}(x) - 
    \E\left[
      \mathds{1}_{\l\{ \bar \theta^\eps\wedge \tau_\varnothing< S_\varnothing\r\}}  
      \rme^{-\gamma \l(\bar\theta^\eps \wedge \tau_\varnothing\r)}
      \left(\varphi_{\varnothing}\left(X^{x_0}_\varnothing(\bar\theta^\eps\wedge \tau_\varnothing)\right)-\zeta\wedge \eta\wedge \delta\wedge \eps\right)
    \right] 
    \\&
    \qquad\qquad\quad
    - \E\left[
      \mathds{1}_{\l\{ \bar \theta^\eps\wedge \tau_\varnothing \geq S_\varnothing\r\}}
      \prod_{\ell=0 }^{\l|
        \Vc_{S_\varnothing}\r|-1}
      \rme^{-\gamma S_\varnothing}
      \Big(\varphi_{\ell}\left(X^{\varnothing}_{S_\varnothing}\right)-\zeta\wedge \eta\wedge \delta\wedge \eps\Big)
    \right]
    \\&=
    \varphi_{\varnothing}(x) - 
    \E\left[
      \Prod_{i\in\Vc_{\bar \theta^\eps\wedge \tau_\varnothing\wedge S_\varnothing}}
      \rme^{-\gamma \l( \bar \theta^\eps\wedge \tau_\varnothing\wedge S_\varnothing\r)}
      \Big(\varphi_{i}\left(X^{\varnothing}_{S_\varnothing}\right)-\zeta\wedge \eta\wedge \delta\wedge \eps\Big)
      \right] \;. 
\end{align*}
Applying Ito's formula to the r.h.s.\ of the previous inequality, we have
\begin{align*}
  &\varphi_{\varnothing}(x_0)-\E\left[\Prod_{i\in\Vc_{\bar \theta^\eps\wedge \tau_\varnothing\wedge S_\varnothing}}
  \rme^{-\gamma \bar \theta^\eps\wedge \tau_\varnothing\wedge S_\varnothing}
  \Big(\varphi_{i}\left(X^{\varnothing}_{S_\varnothing}\right)-\zeta\wedge \eta\wedge \delta\wedge \eps\Big)\right]
  \\
  &=  
  \zeta\wedge \eta\wedge \delta\wedge \eps
  +\E\left[\int_0^{\bar \theta^\eps\wedge \tau_\varnothing\wedge S_\varnothing}-\Lc\big(\varnothing, e^{-\gamma s}\big(\varphi_\cdot-\zeta\wedge \eta\wedge \delta\wedge \eps\big)\big)\l(X^{\varnothing}_s\r)\rmd s\right] \;. 
\end{align*}
Therefore, since
\begin{align*}
  \E\left[
    \int_0^{\bar \theta^\eps\wedge \tau_\varnothing\wedge S_\varnothing}-\Lc\l(
      \varnothing, \rme^{-\gamma s}(\varphi_\cdot-\zeta\wedge \eta\wedge \delta\wedge \eps)
    \r)
  \l(X^{\varnothing}_s\r)\rmd s\right] &\geq & 0\;,
\end{align*}
from \eqref{subsol:inside_L_geq_eta} and the definition of $\bar\theta^\eps$, we can conclude that
\begin{align*}
  &v_{\varnothing}(x_0) - \E\left[
    \prod_{j\in L^\mu_{\theta^\eps}{\setminus D_\tau}  }
  \left(\rme^{-\gamma\theta^\eps_j}
  v_{j}\left(X^{\varnothing}_{\theta^\eps_j}\right)\right)^{\mathds{1}_{\l\{ \theta^\eps_j\leq \tau_j \r\}}}
  \prod_{j\in L^\mu_{\tau}\setminus D_{\theta^\eps}}
  \left( \rme^{-\gamma\tau_j}
  g_{j}\left(X^{\varnothing}_{\tau_j}\right)\right)^{\mathds{1}_{\l\{ \tau_j<\theta^\eps_j \r\}}}\right]  \\ &\qquad\qquad\qquad\qquad\qquad\qquad\qquad\qquad\qquad\qquad\qquad\qquad\qquad\qquad\qquad\geq ~\zeta\wedge \eta\wedge \delta\wedge \eps  \;,
\end{align*}
for any $\tau\in\Sc\Lc$. However, since $\zeta\wedge \eta\wedge \delta\wedge \eps  >  0$, we get a contradiction of \eqref{eq:DPP}.
\end{proof}

\subsection{Proof of Theorem \ref{Thm:comparison}}

Before establishing the comparison principle, we present the following preliminary lemma. In this result, we examine the alterations in the PDE \eqref{eq:DPE} when a multiplicative penalization is applied to the viscosity solutions. Take $C>1$ and $\kappa>0$, to be fixed later, and define $\phi:\R^d\times\Ic\to\R$ as 
\begin{align}
\label{eq:def:phi}
  \phi_i(x):=C^{|i|}(|x|^2 +1)^{\kappa}\eqsp,
\end{align}
together with the following operator
\begin{align*}
  \tilde \Lc: \Ic\times \R^d\times \R\times \R^d\times \S^d\times \R^\N &\to \R\\
  \big(i,x,r, q,M,(r_\ell)_{\ell\in\N}\big)&\mapsto 
  \frac{1}{2}\text{Tr}\left(\sigma(x)\sigma^\top(x) M\right)+ \tilde b(x)^\top q\\
  &\qquad\quad + \alpha(x)\sum_{k\geq0}p_k(x) \frac{\prod_{\ell=0}^{k-1}\phi_{i\ell}(x)}{\phi_{i}(x)}\prod_{\ell=0}^{k-1} r_{\ell} - (\alpha(x)+ \tilde\gamma(x)) r\;,
\end{align*}
with
\begin{align*}
  \tilde b(x) = b(x) + \left( \frac{\sigma \sigma^\top D\phi}{\phi}\right)(x)\qquad
  \tilde \gamma(x) = \gamma - \left(\frac{b^\top D\phi}{\phi}\right)(x)-\frac{1}{2\phi(x)}\text{Tr} 
  \left(\sigma\sigma^\top D^2\phi \right)(x)\;.
\end{align*}
We observe that the function $\tilde b$ and $\tilde \gamma$  do not depend on the variable $i\in\Ic$, although the function $\phi$ does. 
\begin{Lemma}
\label{Lemma:pre-comparison}
  Let $\{u_i\}_{i\in\Ic}$ (resp. $\{v_i\}_{i\in\Ic}$) be a bounded nonnegative continuous viscosity supersolution (resp. subsolution) to \eqref{eq:DPE}. 
  Then, the functions $\{\tilde u_i\}_{i\in\Ic}$ (resp. $\{\tilde v_i\}_{i\in\Ic}$) defined by
  \begin{align*}
    \tilde u_i(x) = \frac{u_i(x)}{\phi_i(x)} \quad\left(\text{resp. } \tilde v_i(x) = \frac{v_i(x)}{\phi_i(x)}\right)\;,\quad x\in\R^d\;,
  \end{align*}
  are bounded nonnegative
  viscosity supersolution (resp. subsolution) to 
  \begin{align}\label{eq:DPE-tilde}
    \min\left\{-\tilde\Lc_i\left(x,\tilde v_i(x), D\tilde v_i(x),D^2\tilde v_i(x), \big(\tilde v_{i\ell}(x)\big)_{\ell\in\N} \right)~;~
    \tilde v_i(x)-\tilde g_i(x)\right\}=0\;,
  \end{align}
  with $\tilde g_i(x)=g_i(x)/\phi(x)$, for $(i,x)\in\Ic\times \R^d$.
\end{Lemma}
\begin{proof}
  We prove the supersolution case, the subsolution case is proven with the same techniques.
  
  Fix $(i_0,x_0)\in\Ic\times\R^d$ and some test functions $\tilde\varphi_i\in C^2(\R^d)$, for $i\in\Ic$
  , and $\bar{\tilde\varphi}\in C^0(\R^d)$ satisfying \eqref{eq:def:visc_sol:sup_phi_bar_phi} 
  and
  \begin{align*}
    0=\left(\tilde u_{i_0}- \tilde\varphi_{i_0}\right)(x_{0}) =  \min_{\Ic\times \mathbb{R}^{d}}\left(\tilde u_{\cdot}-\tilde\varphi_\cdot\right)\;.
  \end{align*}
  Therefore, for $\varphi_i = \phi_i\tilde\varphi_i$, for $i\in\Ic$, with $\phi$ as in \eqref{eq:def:phi}, we have 
  \begin{align*}
    0=\left(u_{i_0}- \varphi_{i_0}\right)(x_{0}) = \min_{\Ic\times \R^d}\left( u_{\cdot}-\varphi_\cdot\right)
    \;.
  \end{align*}
  Moreover, the condition \eqref{eq:def:visc_sol:sup_phi_bar_phi} is satisfied with respect to the function $\bar\varphi=\phi\bar{\tilde\varphi}$.
  Therefore, the functions $(\varphi_i)_{i\in\Ic}$ satisfy \eqref{eq:DPE}. Dividing this equation by the positive function $\phi$ and applying the product rule, we get that the functions $(\tilde\varphi_i)_{i\in\Ic}$ satisfy \eqref{eq:DPE-tilde}.
\end{proof}

\begin{proof}[Proof of Theorem \ref{Thm:comparison}]
  We assume to the contrary that there exists $(z,j)\in \R^d\times \Ic$ such that 
  \begin{align}\label{eq:comparison:absurdHP}
  u_j(z)-v_j(z) \geq ~& \delta \;,
  \end{align} 
  for some $\delta>0$.
  Take $\tilde u_i = u_i/\phi_i$ (resp. $\tilde v_i =v_i/\phi_i$), for $i\in\Ic$, with $\phi$ as in \eqref{eq:def:phi}. We take the constant $C$ in $\phi$ satisfying $C>1$ and \eqref{condC1}-\eqref{condC2}.
  We have in particular that
  \begin{align*}
  \sup_{(i,x)\in\Ic\times\R^d}\tilde u_i(x) \leq 1\;,\qquad
  \sup_{(i,x)\in\Ic\times\R^d}\tilde v_i(x) \leq 1\;.
  \end{align*}
  Since  $u_i$ and $v_i$ are bounded and the constant $\kappa$ in the definition of $\phi$ is strictly positive, we have
  \begin{align}\label{croissance:v-tilde}
    \lim_{(|i|,x)\to\infty} (\tilde u_i + \tilde v_i)(x) =0\;.
  \end{align}
  Combining this with \eqref{eq:comparison:absurdHP}, together with the fact that $\phi>0$, there exists $(i_0,x_0)\in \Ic\times\R^d$ such that
  \begin{align}
  \label{eq:comparison:maximization_pt}
    \bar M_{0+} := \sup_{(i,x)\in \Ic\times\R^d}\tilde u_{i}(x)- \tilde v_{i}(x) = & 
    \tilde u_{i_0}(x_0)- \tilde v_{i_0}(x_0)~~\geq~~ \frac{\delta}{\phi_j(z)}~~>~~ 0\;.
  \end{align}
  For $n\geq1$, consider
  \begin{align*}
    \bar M_{n} = &  \sup_{(i,x,y)\in \Ic\times \R^d\times\R^d}\tilde u_i(x)- \tilde v_{i}(y)
    - \frac{n}{2}|x-y|^2\;.
  \end{align*}
  From \eqref{croissance:v-tilde}, there exists $(i_n,x_n,y_n)$ such that 
  \begin{align*}
    \bar M_{n} = & \tilde u_{i_n}(x_n)- \tilde v_{i_n}(y_n)
    - \frac{n}{2}|x_n-y_n|^2\;.
  \end{align*}
  From the definition of $\bar M_{n}$, taking $x=y$ in the previous supremum, we obtain 
  \begin{align} \label{eq:comparison:sup_geq_0}
    0<~\frac{\delta}{\phi_j(z)} \leq ~ \bar M_{0+}\leq~ \bar M_{n}\leq~ 2\;.
  \end{align}
  This yields 
  \begin{align}\label{eq:comparison:bound_x_y_n}
    \frac{n}{2}|x_n-y_n|^2 \leq & 2\;.
  \end{align}

  Using \eqref{eq:comparison:sup_geq_0} and \eqref{croissance:v-tilde}, up to  a sub-sequence, we can take $i_n=i^*$, for some $i^*\in\Ic$ and all $n$, and $(x_n,y_n)\to(x^*,y^*)$, as $n\rightarrow\infty$. Therefore, \eqref{eq:comparison:bound_x_y_n} yields
  \begin{align*}
    \lim_{n\to\infty}|x_n-y_n|=~0 & ~\mbox{ and }~   x^*=y^*\;.
  \end{align*}
  Moreover, from \eqref{eq:comparison:sup_geq_0}, we obtain
  \begin{align*}
  \lim_{n\to\infty}\frac{n}{2}|x_n-y_n|^2 = & 0\;.
  \end{align*}
  Without loss of generality, we can take the maximization point in \eqref{eq:comparison:maximization_pt} to be $(i^*,x^*)$, $i.e.$, $(i_0,x_0)=(i^*,x^*)$. Therefore, as $(x_n, y_n)\in \R^d\times\R^d$ is a maximizer of $\bar M_{n}$, we may apply Ishii's lemma \citep[see, $e.g.$, Theorem 8.3,][]{crandall1992users} and Lemma \ref{Lemma:pre-comparison}. Therefore, there exist $A_n, B_n\in \S^d$ such that
  \begin{align*}
    \min\left\{-\tilde \Lc_{i_0}\left(x_n,\tilde u_{i_0}(x_n), n(x_n-y_n), A_n, \big(\tilde u_{i_0\ell}(x_n)\big)_{\ell\in\N} \right)~;~
    \tilde u_{i_0}(x_n)-\tilde g_{i_0}(x_n)\right\} &\leq~0\eqsp,\\
    \min\left\{-\tilde \Lc_{i_0}\left(y_n,\tilde v_{i_0}(y_n), n(x_n-y_n), B_n, \big(\tilde v_{i_0\ell}(y_n)\big)_{\ell\in\N} \right)~;~
    \tilde v_{i_0}(y_n)-\tilde g_{i_0}(y_n)\right\}&\geq~0\;,
  \end{align*}
  and
  \begin{align*}
    -3n ~
      \I_{2d} \leq
    \begin{pmatrix}
      A_n & 0\\0 & - B_n
    \end{pmatrix} \leq ~& 3n \begin{pmatrix}
      \I_d & -\I_d\\-\I_d & \I_d
    \end{pmatrix}
    \;.
  \end{align*}

  If there exists a subsequence of $\{x_n\}_n$, still denoted $\{x_n\}_n$, such that $\tilde u_{i_0}(x_n)-\tilde g_{i_0}(x_n)\leq 0$, we would get
  \begin{align*}
    \left[\tilde u_{i_0}(x_n)-\tilde g_{i_0}(x_n)\right] - \left[\tilde v_{i_0}(y_n)-\tilde g_{i_0}(y_n)\right] &\leq ~ 0\;,
  \end{align*}
  for any $n$. This is, however, in contradiction with \eqref{eq:comparison:sup_geq_0}, the fact that $(x_n,y_n)\to (x_0,x_0)$ and the definition of $(i_0,x_0)$. Therefore, we must have
  \begin{align}
  \label{eq:comparison:eq1}
    -\tilde \Lc_{i_0}\left(x_n,\tilde u_{i_0}(x_n), n(x_n-y_n), A_n, \big(\tilde u_{i_0\ell}(x_n)\big)_{\ell\in\N} \right) & \leq ~0\;,\\
  \label{eq:comparison:eq2}
    -\tilde \Lc_{i_0}\left(y_n,\tilde v_{i_0}(y_n), n(x_n-y_n), B_n, \big(\tilde v_{i_0\ell}(y_n)\big)_{\ell\in\N} \right) &\geq~0\;,
  \end{align}
  for $n$ large enough. 
  Using a telescopic sum and the fact that the functions $u_i$ and $v_i$ are positive and bounded by $C$, we have
  \begin{align*}
    & \sum_{k\geq0}p_k (x_n)\frac{\prod_{\ell=0}^{k-1}\phi_{i\ell}(x_n)}{\phi_{i}(x)}\prod_{\ell = 0}^{k-1} \tilde  u_{i_0\ell}(x_n) -\sum_{k\geq0}p_k (y_n)\frac{\prod_{\ell=0}^{k-1}\phi_{i\ell}(y_n)}{\phi_{i}(y_n)} \prod_{\ell = 0}^{k-1} \tilde v_{i_0\ell}(y_n)
    \\
    &= 
    \Delta_n+\sum_{k\geq0}p_k(x_n)\sum_{\ell = 0}^{k-1}\frac{\phi_{i\ell}(x_n)}{\phi_{i}(x_n)}
      \left(\prod_{\bar\ell = 0}^{\ell-1}u_{i_0\bar\ell}(x_n)\right)\left( \tilde u_{i_0\ell}(x_n)-\tilde v_{i_0\ell}(y_n)\right)
      \left(\prod_{\bar\ell = \ell+1}^{k-1}v_{i_0\bar\ell}(y_n)\right)
      \eqsp,
  \end{align*}
  with
  \begin{align*}
    \Delta_n:=\sum_{k\geq0}\l(p_k (x_n)-p_k (y_n)\r)\frac{\prod_{\ell=0}^{k-1}\phi_{i\ell}(y_n)}{\phi_{i}(y_n)} \prod_{\ell = 0}^{k-1} \tilde v_{i_0\ell}(y_n)
    \eqsp.
  \end{align*}
  Using that $(i_0,x_n,y_n)$ is a maximizer of $\bar M_{n}$ and \eqref{eq:comparison:sup_geq_0}, we have
  \begin{align*}
    & \sum_{k\geq0}p_k (x_n)\frac{\prod_{\ell=0}^{k-1}\phi_{i\ell}(x_n)}{\phi_{i}(x)}\prod_{\ell = 0}^{k-1} \tilde  u_{i_0\ell}(x_n) -\sum_{k\geq0}p_k (y_n)\frac{\prod_{\ell=0}^{k-1}\phi_{i\ell}(y_n)}{\phi_{i}(y_n)} \prod_{\ell = 0}^{k-1} \tilde v_{i_0\ell}(y_n)
    \\
    & \leq \Delta_n+\sum_{k\geq0}p_k(x_n)\sum_{\ell = 0}^{k-1}\frac{\phi_{i\ell}(x_n)}{\phi_{i}(x_n)}
    \left(\prod_{\bar\ell = 0}^{\ell-1}u_{i_0\bar\ell}(x_n)\right)\left( \tilde u_{i_0}(x_n)-\tilde v_{i_0}(y_n)\right)
    \left(\prod_{\bar\ell = \ell+1}^{k-1}v_{i_0\bar\ell}(y_n)\right)
  \end{align*}
  On the one hand, since
  \begin{align*}
  |\Delta_n| & \leq \sum_{k\geq0}|p_k(x_n)-p_k(y_n)|C^k\;,\qquad\text{ for } n\geq 1\;,
  \end{align*}
  applying Proposition \ref{prop:results:p_k} (ii), we have $\Delta_n\to0$ as $n\to+\infty$.
  We get
  \begin{align*}
    & \sum_{k\geq0}p_k (x_n)\frac{\prod_{\ell=0}^{k-1}\phi_{i\ell}(x_n)}{\phi_{i}(x)}\prod_{\ell = 0}^{k-1} \tilde  u_{i_0\ell}(x_n) -\sum_{k\geq0}p_k (y_n)\frac{\prod_{\ell=0}^{k-1}\phi_{i\ell}(y_n)}{\phi_{i}(y_n)} \prod_{\ell = 0}^{k-1} \tilde v_{i_0\ell}(y_n)
    \\
    &\leq  \Delta_n+   \l(\sum_{k\geq0}p_k(x_n) C^{k}\frac{C^k-1}{C-1}\r)\left(\tilde u_{i_0}(x_n)-\tilde v_{i_0}(y_n)\right)
    \;.
  \end{align*}

  Since $C> 1$, from Assumption \ref{Assumption_H_0BIS}, 
  we have
  \begin{align*}
    \sum_{k\geq0}p_k(x_n) C^{k}\frac{C^k-1}{C-1}
    &\leq
    \frac{1}{C-1}\sum_{k\geq0}p_k(x_n) C^{2k}
    \leq \frac{1}{C-1}\sum_{k\geq0}p_k(x_n) \exp\l(2k\log(C)\r)
    \\
    & \leq\frac{1}{C-1}
    \sum_{k\geq0}p_k(x_n) \sum_{\ell\geq 0}\frac{\l(2k \log(C)\r)^\ell}{\ell!}
    \\
    & \leq\frac{1}{C-1}
    \sum_{\ell\geq 0}\frac{\l(\log(C)\r)^\ell}{\ell!} 2^\ell M_{\ell}
    \leq \bar{M}\frac{C}{C-1}\;.
  \end{align*}
  Note that
  \begin{align*}
    \frac{D\phi(x)}{\phi(x)} &= \frac{2\kappa x}{|x|^2+1}\;,
    \\
    \frac{D^2\phi(x)}{\phi(x)} &= 4\kappa(\kappa-1)\frac{xx^\top}{(|x|^2+1)^2}+2\kappa \frac{\I_d}{|x|^2+1}\;,
  \end{align*}
  for $x\in\R^d$. Therefore,
  $\tilde b$ is locally Lipschitz and $\tilde\gamma-\gamma$ is equal to a bounded function in $\R^d$ multiplied by $\kappa$. This means that there exists $\kappa$ small enough such that
  \begin{align*}
    \tilde\gamma(x)- \alpha(x)\l(
        \bar{M}\frac{C}{C-1}-1
    \r) \geq~ \frac{1}{2}\l(
        \gamma-\bar \alpha\l(
        \bar{M}\frac{C}{C-1}-1
    \r)
    \r)& > 0\;,
  \end{align*}
  for all ${x}$ in the neighbourhood of $x_0$.  Applying then \eqref{eq:comparison:eq1}-\eqref{eq:comparison:eq2}, we get, for $n$ large enough,
  \begin{align*}
  (\tilde \gamma(x_n) - \alpha(x_n)(M-1))\tilde u_{i_0}(x_n)-\l(\tilde \gamma(y_n) - \alpha(x_n)\l(
        \bar{M}\frac{C}{C-1}-1
    \r)\r)\tilde v_{i_0}(y_n) \leq &\\
  \left(\tilde b(x_n)- \tilde b(y_n)\right)^\top n\left( x_n- y_n\right) + 
  \frac{1}{2}\text{Tr}\left(\sigma\sigma^\top(x_n) A_n - \sigma\sigma^\top(y_n) B_n\right)&\;.
  \end{align*}
  Sending $n$ to infinity, we obtain
  \begin{align*}
  0\geq \l(\tilde \gamma(x_0) - \alpha(x_0)\l(
        \bar{M}\frac{C}{C-1}-1
    \r)\r)\left(\tilde u_{i_0}(x_0) - \tilde v_{i_0}(x_0)\right)\;.
\end{align*}
However, from \eqref{condC2} and for $\kappa$ small enough, the previous equation is in contradiction to \eqref{eq:comparison:absurdHP}.

\end{proof}

\bibliography{Bib}{}
\bibliographystyle{apalike}

\end{document}

%% file: notation.tex
\newtheorem{Theorem}{Theorem}[part]
\newtheorem{Definition}{Definition}[part]
\newtheorem{Proposition}{Proposition}[part]
\newtheorem{Assumption}{Assumption}[part]
\newtheorem{Lemma}{Lemma}[part]

\newtheorem{Corollary}{Corollary}[part]
\newtheorem{Remark}{Remark}[part]

\def \Prod{\displaystyle\prod}

\def \I{\mathbb{I}}

\def \N{\mathbb{N}}
\def \R{\mathbb{R}}

\def \E{\mathbb{E}}
\def \F{\mathbb{F}}

\def \P{\mathbb{P}}

\def \D{\mathbb{D}}
\def \S{\mathbb{S}}

\def \1{\mathds{1}}
\def \0{\mathds{O}}

\def \Bc{\mathcal{B}}
\def \Cc{\mathcal{C}}

\def \Fc{\mathcal{F}}

\def \Hc{\mathcal{H}}
\def \Ic{\mathcal{I}}

\def \Lc{\mathcal{L}}
\def \Mc{\mathcal{M}}
\def \Nc{\mathcal{N}}

\def \Sc{\mathcal{S}}

\def \Vc{\mathcal{V}}

\newcommand{\restr}[2]{#1_{\mkern 1mu \vrule height 2ex\mkern2mu #2}}

\usepackage[textwidth=3cm]{todonotes}

\def \l{\left}
\def \r{\right}
\def \eps{\varepsilon}

\def\beqs{\begin{eqnarray*}}
\def\enqs{\end{eqnarray*}}
\def\beq{\begin{eqnarray}}
\def\enq{\end{eqnarray}}

\def\i{\mbox{\tiny \rm 1}}

\def\S{\mathbf S}

\providecommand{\keywords}[1]
{
  \small	
  \textbf{Keywords---} #1
}

\def\rme{{\rm e}}
\def\rmd{{\rm d}}

\def\eqsp{\;}